\documentclass[11pt,a4paper,reqno]{amsart}
\usepackage{amsmath}
\usepackage{amsfonts}
\usepackage{amssymb}
\usepackage{amscd}
\usepackage{url}
\usepackage{enumerate}
\usepackage[pdftex,bookmarks=true]{hyperref}

\newcommand\Vol{{\operatorname{Vol}}}

\newcommand\R{{\mathbf{R}}}
\newcommand\C{{\mathbf{C}}}

\renewcommand\P{{\mathbf{P}}}
\newcommand\E{{\mathbf{E}}}
\newcommand\bv{{\mathbf{v}}}

\newcommand\dist{{\operatorname{dist}}}
\newcommand\Z{{\mathbf{Z}}}

\newcommand\F{{\mathbf{F}}}
\newcommand\I{{\mathbf{I}}}

\newcommand\row{{\mathbf{r}}}

\newcommand\eps{\varepsilon}

\newcommand\Ba{{\mathbf a}}

\newcommand\Be{{\mathbf e}}

\newcommand\Bu{{\mathbf u}}
\newcommand\Bv{{\mathbf v}}

\newcommand\Bx{{\mathbf x}}

\newcommand\BB{{\mathbf B}}

\newcommand\BN{{\mathbf N}}

\newcommand\Ber{\it{Ber}}
\newcommand\Gau{\it {Gau}}
\newcommand\sym{\it {sym}}

\newcommand\ep{\epsilon}
\newcommand\LCD{\mathbf{LCD}}
\newcommand\Cor{\mathbf{Cor}}
\newcommand\Gr{\mathbf {Gr}}
\newcommand\x{\xi}

\theoremstyle{plain}
  \newtheorem{theorem}[subsection]{Theorem}
  \newtheorem{conjecture}[subsection]{Conjecture}
  \newtheorem{problem}[subsection]{Problem}

  \newtheorem{proposition}[subsection]{Proposition}
  \newtheorem{fact}[subsection]{Fact}
  \newtheorem{lemma}[subsection]{Lemma}
  \newtheorem{corollary}[subsection]{Corollary}
  \newtheorem{question}[subsection]{Question}
  \newtheorem{example}[subsection]{Example}
  
  \newtheorem{remark}[subsection]{Remark}

\theoremstyle{definition}
  \newtheorem{definition}[subsection]{Definition}

\begin{document}

\title[Small ball probability, Inverse  theorems, and applications]{Small probability, Inverse  theorems,  and applications}


\author{Hoi H. Nguyen}
\email{hoi.nguyen@yale.edu}

\author{Van H. Vu}
\email{van.vu@yale.edu}

\address{Department of Mathematics, Yale University, 10 Hillhouse Ave., New Haven, CT 06511}

\thanks{The first author is supported by research grant DMS-1256802}
\thanks{The second author is supported by research grants DMS-0901216 and AFOSAR-FA-9550-12-1-0083}


\maketitle
\begin{abstract}
Let $\xi$ be a real random variable with mean zero and variance one 
 and  $A=\{a_1,\dots,a_n\}$ be a multi-set in $\R^d$. The random sum 

$$S_A := a_1 \xi_1 + \dots + a_n \xi_n $$ where $\xi_i$ are iid copies of $\xi$ is of fundamental importance  in probability and its applications. 

We discuss the {\it small ball} problem,  the aim of which is to  estimate the maximum probability that $S_A$ belongs to a ball with given small radius, following 
the discovery made by Littlewood-Offord and Erd\H{o}s almost 70 years ago. We will mainly focus on 
recent developments that characterize the structure of those sets $A$ where the small ball probability is relatively large.  Applications of these results 
include full solutions or significant progresses of many open problems in different areas. 

\end{abstract}

\tableofcontents 

\section{Littlewood-Offord and Erd\H os estimates }\label{section:introduction}

Let $\xi$ be a real random variable with mean zero and variance one 
 and  $A=\{a_1,\dots,a_n\}$ be a multi-set in $\R$ (here $n \rightarrow \infty$). The random sum 

$$S_A := a_1 \xi_1 + \dots + a_n \xi_n $$ where $\xi_i$ are iid copies of $\xi$ plays an essential  role in probability.   The Central Limit Theorem, arguably the most important theorem in the field, asserts that if the $a_i$'s are the same, then 

$$ \frac{S_A} {\sqrt {  \sum_{i=1}^n |a_i| ^2 }  } \longrightarrow \BN(0,1) . $$ 

Furthermore, Berry-Ess\'een theorem shows that if $\xi$ has bounded third moment,  then the rate of convergence is $O(n^{-1/2} )$. 
This, in particular, implies that for any small open interval $I$ 

$$\P ( S_A \in I ) = O(|I| / n^{1/2} ). $$

The assumption that the $a_i$'s are the same are, of course, not essential. Typically, it suffices to assume that none of 
the $a_i$'s is dominating; see \cite{Fellerbook} for more discussion. 

The probability $\P ( S_A \in I)$ (and its high dimensional generalization) will be referred to as {\it small ball } probability throughout the paper. In 1943, Littlewood and Offord, in connection with their studies of 
random polynomials \cite{LO}, raised the problem of estimating the small probability for {\it arbitrary} coefficients $a_i$. Notice that when we do not assume anything about the $a_i$'s, even the Central Limit Theorem may fail, so
 Berry-Ess\'een type bounds no longer apply. Quite remarkably, Littlewood and Offord managed to show

\begin{theorem} \label{theorem:LO}  If $\xi$ is  Bernoulli (taking values $\pm 1$ with probability $1/2$)  and 
$a_i$ have absolute value at least 1, then for any open  interval $I$ of length 2,

$$\P (S_A  \in I) = O( \frac{\log n}{n^{1/2}} ). $$ \end{theorem}

Shortly after  Littlewood-Offord result, Erd\H{o}s \cite{E} gave a beautiful combinatorial proof of the following refinement, which turned out to be sharp.

\begin{theorem}\label{theorem:Erdos} Under the assumption of Theorem \ref{theorem:LO} 

\begin{equation}\label{eqn:Erdos}
\P(S_A \in I) \le \frac{ \binom{n}{\lfloor n/2\rfloor}}{2^n } = O(\frac{1}{n^{1/2} }). 
\end{equation}
\end{theorem}

\begin{proof} (of Theorem \ref{theorem:Erdos}) Erd\H os' 
proof made an ingenious use of Sperner's lemma, which asserts that if $\mathcal {F} $ is an anti-chain on a set of $n$ elements, then $\mathcal{F}$ has at most 
${ n \choose {\lfloor n/2 \rfloor}} $ elements (an anti-chain is a family of subsets none of which contains the other). Let $x$ be a fixed number. By reversing the sign of $a_i$ if necessary, one can assume that $a_i\ge 1$ for all $i$. Now let $\mathcal{F}$ be the set of all subsets $X$ of $[n] :=\{1,2 \dots, n\} $ such that 

$$\sum_{i\in X} a_i - \sum_{j \in \bar{X}}a_j \in (x-1,x+1).$$ 

One can easily verify that $\mathcal{F}$ is an anti-chain. Hence, by Sperner's lemma, 

$$|{\mathcal{F}}| \le \frac{\binom{n}{n/2}}{2^n},$$ completing the proof. \end{proof} 


The problem was also studied in probability by Kolmogorov, Rogozin, and others;  we refer the reader to \cite{Kol1,Kol2} and \cite{Rog}. Erd\H{o}s' result is popular in 
the combinatorics community and has became the starting point for a whole theory that we now start to discuss.


\vskip2mm 

{\it Notation.} We use the asymptotic notation such as $O, o, \Theta$ under the assumption that $n \rightarrow \infty$; $O_\alpha (1)$ means the constant in big $O$ depends on $\alpha$. All logarithms have natural base, if not 
specified otherwise.

\section{High dimensional extenstions}  \label{section:HD}

Let $\xi$ be a real random variable and $A=\{a_1, \dots, a_n\}$ a multi-set in $\R^d$, where $d$ is fixed. 
 For a given radius $R >0$, we define 

$$\rho_{d, R,\xi}(A):=\sup_{x \in \R^d} \P\big(a_1\xi_1+\dots+a_n \xi_n \in \BB(x, R) \big),$$

where  $\xi_i$ are iid copies of $\xi$, and $\BB(x,R)$ denotes the open disk of radius $R$ centered at $x$ in $\R^d$.  
Furthermore, let 

$$p (d,  R, \xi, n) := \sup_{A} \rho_{d,R, \xi} (A) $$ where $A$ runs over all multi-sets of size $n$ in $\R^d$ consisting of vectors with norm at least 1. 
Erd\H{o}s' theorem can be reformulated as 

$$p(1, 1, \Ber, n)  =  \frac{\binom{n}{\lfloor n/2 \rfloor}}{2^n} = O(n^{-1/2}) . $$

In the case $d=1$,  Erd\H{o}s  obtained the optimal  bound for any fixed  $R$. In what follows we define $s:=\lfloor R \rfloor +1$.

\begin{theorem} \label{Erd1} Let $S(n,m)$ denote the sum of the largest $m$ binomial coefficients
$\binom{n}{i}, 0 \le i \le n$. Then

\begin{equation} \label{ERD1} p (1, R, \Ber, n)   = 2^{-n} S(n, s). \end{equation} 
\end{theorem}

The case $d \ge 2$ is much more complicated and has been studied by many researchers. In particular, Katona \cite{Kat} and Kleitman \cite{Kle1} showed that $p(2,1,\Ber,n)= 2^{-n}\binom{n}{\lfloor n/2 \rfloor}$. This result was extended by Kleitman \cite{Kle2} to arbitrary dimension $d$,  

\begin{equation}\label{eqn:Kle:1}
p(d,1,\Ber,n)= \frac{\binom{n}{\lfloor n/2 \rfloor}}{2^n}.
\end{equation}

The estimate for general radius $R$ is much harder. In \cite{Kle3}, Kleitman showed that $2^np(2,R,\Ber,n)$ is bounded from above by the sum of the $2 \lfloor R/\sqrt{2} \rfloor$ largest binomial coefficients in $n$. For general $d$, Griggs \cite{Griggs1} proved that 

$$p(d,R,\Ber,n)\le 2^{2^{d-1}-2} \lceil R \sqrt{d} \rceil \frac{\binom{n}{\lfloor n/2\rfloor }}{2^n}.$$

This result was  then improved by Sali \cite{S1,S2} to     

$$p(d,R,\Ber,n)\le 2^{d} \lceil R \sqrt{d} \rceil \frac{\binom{n}{\lfloor n/2\rfloor }}{2^n}.$$

A major improvement is due to Frankl and F\"uredi \cite{FF}, who proved

\begin{theorem}\label{FF1} For any fixed $d$ and $R$
\begin{equation}\label{panda}
p(d, R, \Ber, n) = (1+o(1)) 2^{-n} S(n,s). 
\end{equation} \end{theorem}

This result is asymptotically sharp. In view of \eqref{ERD1} and \eqref{eqn:Kle:1},  it is  natural to ask if 
 the exact estimate
 \begin{equation} \label{panda-2} p(d, R, \Ber, n) = 2^{-n} S(n,s), \end{equation}  
holds for all fixed dimension $d$.
  However, this has turned out to be false.  The authors of  \cite{Kle2, FF}  observed that \eqref{panda-2}  fails if $s \geq 2$ and
\begin{equation}\label{kappas}
 R > \sqrt{(s-1)^2+1} .
\end{equation}

\begin{example}
Take $v_1= \dots =v_{n-1} =\Be_1$ and $v_n=\Be_2$, where $\Be_1, \Be_2$ are two orthogonal unit vectors. For this system, let $B$ be the ball of radius $R$ centered at $v=(v_1+\dots+v_{n})/2$. Assume that $n$ has the same parity with $s$, then by definition we have 

$$\P(S_V\in \BB(v,R))= 2 \sum_{(n-s)/2 \le i \le (n+s)/2}\binom{n-1}{i}/2^n > 2^{-n} S(n,s).$$
\end{example}

Frankl and F\"uredi raised the following problem. 

\begin{conjecture} \cite[Conjecture 5.2]{FF} \label{conj:FF}
Let $R,d$ be fixed. If  $s-1 \le  R  < \sqrt {(s-1)^2 +1}$ and 
$n$ is sufficiently large, then 

$$p(d, R, \Ber, n) = 2^{-n} S(n,s). $$
\end{conjecture} 

The conjecture has been confirmed for $s=1$ by Kleitman (see \eqref{eqn:Kle:1}) and for $s=2,3$ by Frankl and F\"uredi \cite{FF} (see \cite[Theorem 1.2]{FF}).  Furthermore, Frankl and F\"uredi  showed that 
\eqref{panda-2} holds under a stronger assumption that 
$s-1 \le  R \le  (s-1) + \frac{1}{10s^2}$. A few years ago, Tao and the second author  proved Conjecture \ref{conj:FF} for  $s \ge 3$.  This, combined with the above mentioned earlier results, established 
the conjecture in full generality \cite{TV:FF}. 

\begin{theorem} \label{TV-FF} Let $R,d$ be fixed. Then there exists a positive number $n_0=n_0(R,d)$ such that the following holds for all $n\ge n_0$ and $s-1 \le  R  < \sqrt {(s-1)^2 +1}$ 

$$p(d, R, \Ber, n) = 2^{-n} S(n,s). $$
\end{theorem} 

We will present a short proof of Theorems \ref{FF1} and \ref{TV-FF} in Section \ref{section:FF}.


\section{Refinements by restrictions on $A$ }  \label{section:restriction}

A totally different  direction  of research started with the observation that  the upper bound in 
\eqref{eqn:Erdos} improves   significantly if we make some extra assumption on the additive structure of  $A$. 
In this section, it is more natural to present the results in discrete form. In the discrete setting, one  considers the probability that $S_A$ takes  a single value (for instance, $\P (S_A=0)$).

 Erd\H os's result in the first section implies

\begin{theorem}\label{theorem:Erdos1}
Let $a_i$ be non-zero real numbers, then

$$\sup_{x \in \R} \P ( S_A =x)\le \frac{\binom{n}{\lfloor n/2 \rfloor}}{2^n} = O(n^{-1/2}) . $$
\end{theorem} 

 Erd\H{o}s and Moser  \cite{EM} showed that under the condition that the $a_i$ are different, the bound improved significantly. 

\begin{theorem}\label{theorem:SSz1}  Let $a_i$ be distinct real numbers, then 

$$\sup_{x \in \R} \P ( S_A =x) =  O(n^{-3/2} \log n) . $$
\end{theorem}

They conjectured that the $\log n$ term is not necessary. 
S\'ark\H{o}zy and Szemer\'edi's  \cite{SSz} confirmed this conjecture 

\begin{theorem}\label{theorem:SSz1}  Let $a_i$ be distinct real numbers, then 

$$\rho_A:= \sup_{x \in \R} \P ( S_A =x) =  O(n^{-3/2}) . $$
\end{theorem}

In \cite{Stan}, Stanley found a different (algebraic) proof for a more precise result, using the hard-Lepschetz theorem from algebraic geometry. 

\begin{theorem}[Stanley's theorem] \label{theorem:Stan}
Let $n$ be odd and $A_0 :=\big\{- \frac{n-1}{2}, \dots, \frac{n-1}{2}  \big \}$. Let $A$ be any set of $n$ distinct real numbers, then

 $$\rho(A):= \sup_{x\in \R} \P(S_A=x) \le \sup_{x\in \R} \P(S_{A_0}=x).$$
\end{theorem}

A similar result holds for the case $n$ is even, see \cite{Stan}. Later, Proctor \cite{Proctor} found a simpler proof for Stanley's theorem. 
 His proof is also algebraic, using tools from Lie algebra. It is interesting to see whether 
 algebraic approaches can be used to obtain {\it continuous}  results.  (For the continuous version of Theorem \ref{theorem:SSz1}, see Section \ref{section:Halasz}.) 

{\it A hierarchy of bounds.} 
We have seen that the Erd\H{o}s' bound of $O(n^{-1/2})$ is sharp, if we allow the $a_i$ to be the same. If we forbid this, then the next bound is 
$O(n^{-3/2})$, which can be attained if the $a_i$ form an arithmetic progression. Naturally, one would ask what happen if we forbid 
the $a_i$ to form an arithmetic progression and so forth.  Hal\'asz' result, discussed in  Section \ref{section:Halasz} , gives 
a satisfying  answer to this question.

\begin{remark}\label{remark:Freiman} 
To conclude this section, let us mention that while  discrete theorems such as Theorem \ref{theorem:Stan} are formalized for 
real numbers, it holds for any infinite abelian groups, thanks to a general trick called Freiman isomorphism (see \cite{TVbook} and also Appendix \ref{section:discrete:proof}). 
In particular, this trick allows us to assume that the $a_i$'s are integers in the proofs.  Freiman isomorphism, however, is not always applicable in 
continuous settings.  \end{remark}


\section{Littlewood-Offord type bounds for higher degree polynomials}  \label{section:quadratic1} 

For simplicity, we present  all  results in this section  in discrete form. The extension to continuous setting is rather straightforward, and thus omitted.

One can view the sum  $S=a_1 \x_1+\dots+a_n \x_n$ as a linear function of the random variables $\x_1,\dots, \x_n$. It is natural to study 
general polynomials of higher degree $k$. Let us first consider the case $k=2$. Following \cite{CTV}, we refer to it as  the Quadratic Littlewood-Offord problem. 

Let $\x_i$ be iid Bernoulli random variables, let $A=(a_{ij})$ be an $n\times n$ symmetric matrix of real entries. We define the {\it quadratic concentration probability} of $A$ by

$$\rho_q(A):= \sup_{a \in \R}\P(\sum_{i,j} a_{ij}\x_i\x_j=a).$$

Similar to the problem considered by Erd\H{o}s and Littlewood-Offord, we may ask what upper bound one can prove for $\rho_q(A)$ provided that 
 the entries $a_{ij}$ are non-zero? This question was first addressed by Costello, Tao and the second author in \cite{CTV}, motivated by their study of 
Weiss' problem  concerning the singularity of a random symmetric matrix (see Section \ref{section:singularity1}).

\begin{theorem}\label{theorem:LO:CTV} Suppose that $a_{ij}\neq 0$ for all $1\le i, j\le n$. Then

$$\rho_q(A)= O(n^{-1/8}).$$ 
\end{theorem}

 The key to the proof of Theorem \ref{theorem:LO:CTV} is  a decoupling lemma, 
which can be proved  using Cauchy-Schwarz inequality. The reader may consider this lemma an exercise, or
consult \cite{CTV} for details.

\begin{lemma}[Decoupling lemma]\label{lemma:decoupling} Let $Y$ and $Z$ be
 random variables and $E=E(Y,Z)$ be an event depending on $Y$ and
$Z$. Then

$$\P( E(Y,Z)) \le\P( E(Y,Z) \wedge E(Y',Z) \wedge E(Y,Z')\wedge
E(Y',Z'))^{1/4} $$

\noindent where $Y'$ and $Z'$ are independent copies of $Y$ and $Z$,
respectively.  Here we use $A \wedge B$ to denote the event that $A$ and $B$ both hold.
\end{lemma}

 Consider the quadratic form $Q(x):=\sum_{ij} a_{ij}\x_i\x_j$, and fix a non-trivial partition $\{1, \dots, n\} =
U_1 \cup U_2$ and a non-empty subset $S$ of $ U_1$. For instance one can take $U_1$ to be the first half of the indices 
and $U_2$ to be the second half. Define $Y :=(\x_i)_{i \in U_1}$ and $Z
:=(\x_i)_{i \in U_2}$. We can write $Q(x)=Q(Y,Z)$. Let
$\x_i'$ be an independent copy of $\x_i$ and set $Y' :=(\x'_i)_{i \in U_1}$
 and $Z' := (\x'_i)_{i \in U_2})$.  By  Lemma \ref{lemma:decoupling}, for any number $x$
 
$$ \P ( Q(Y,Z)= x) \le \P( Q(Y,Z)=Q(Y,Z')=Q(Y',Z)= Q(Y',Z')=x )^{1/4} .$$

On the other hand, if $Q(Y,Z)=Q(Y,Z')=Q(Y',Z)= Q(Y',Z')=x$ then 
regardless the value of $x$ $$R := Q(Y,Z) - Q(Y',Z)- Q(Y,Z') + Q(Y',Z')=0. $$  Furthermore, we can write $R$ as 

$$ R = \sum_{i \in U_1} \sum_{j \in U_2}
 a_{ij} (\x_i-\x_i') (\x_j-\x_j') = \sum_{i \in U_1} R_i w_i, $$ 
 
\noindent where $w_i$ is the random variable $w_i := \x_i - \x'_i$, and $R_i$ is the random variable $\sum_{j \in U_2} a_{ij} w_j$.

We now  can conclude the proof by applying  Theorem \ref{theorem:Erdos1} twice. First, combining this theorem with a combinatorial argument, one can show that (with high probability), many $R_i$ are non-zero. Next, one can condition on the non-zero $R_i$ and apply Theorem \ref{theorem:Erdos1} for the linear form $\sum_{i \in U_1} R_i w_i$ to obtain a bound on $\P(R=0)$. 

The upper bound $n^{-1/8}$ in Theorem \ref{theorem:LO:CTV} can be easily improved to $n^{-1/4} $. The optimal 
bound was obtained by  Costello \cite{C} using, among others, the inverse theorems 
from Section \ref{section:discrete}.

\begin{theorem}[Quadratic Littlewood-Offord inequality]\label{theorem:LO:C} Suppose that  $a_{ij}\neq 0$, $1\le i, j \le n$. Then

$$\rho_q(A)\le n^{-1/2+o(1)}.$$ 
\end{theorem}

The exponent $1/2+o(1)$ is best possible (up to the $o(1)$ term) as demonstrated by the  quadratic form $\sum_{i,j}\x_i\x_j = (\sum_{i=1}^n \xi_i)^2$. 
Both Theorems \ref{theorem:LO:CTV} and 
\ref{theorem:LO:C} hold in a  general setting where the $\xi_i$ are not necessary Bernoulli and only a fraction of the $a_{ij}$'s are non-zero. 

One can extend the argument above  to give bounds of the form $n^{-c_k}$ for a general polynomial of degree $k$. However, due to the repeated use of the decoupling lemma,  $c_k$ decreases 
very fast with $k$. 

\begin{theorem}
Leet $f$ be  a multilinear polynomial of real coefficients in $n$ variables $\xi_1, \dots, \xi_n$  with $m \times n^{k-1}$ monomials of maximum degree $k$. If $\xi_i$ are iid 
Bernoulli random variables, then for any value $x$ 

$$\P(f=x) = O\big( m^{- \frac{1}{2^{(k^2+k) /2} }}\big). $$ 

\end{theorem}

 By a more refined analysis,  Razborov and Viola \cite{RaVi} recently  obtained a better exponent of order  roughly $\frac{1}{2^k}$ (see Section \ref{section:boolean}). 
On the other hand, it might be the case that the bound $n^{-1/2+o(1)}$ holds 
for all degrees $k \ge 2$, under some reasonable assumption on the coefficients of the polynomial. 

Quadratic (and higher degree) Littlewood-Offord bounds play important roles in the study of random symmetric matrices and 
Boolean circuits.  We will discuss these applications in Sections \ref{section:singularity1} and \ref{section:boolean}, respectively.

\section{Application: Singularity of random Bernoulli matrices}  \label{section:singularity1}

Let $M_n$ be a random matrix of size $n$ whose entries are iid Bernoulli random variables. A notorious open problem in 
probabilistic combinatorics is to estimate $p_n$, the probability that $M_n$ is singular  (see \cite{KKSz, TVsing} for more details). 

\begin{conjecture}  \label{notorious} $p_n =(1/2+o(1)) ^n $. \end{conjecture} 

To give the reader a feeling about how  the Littlewood-Offord problem 
can be useful in estimating $p_n$, let us consider the following process.
We expose the rows of $M_n$ one by one from the top. Assume that the first $n-1$ rows are 
linearly independent and form a hyperplane with normal vector $\bv=(a_1, \dots, a_n)$. 
Conditioned on these rows, the probability that 
$M_n$ is singular  is 
$$\P ( X \cdot \bv =0) = \P (a_1 \xi_1 + \dots + a_n \xi_ n =0) , $$ where 
$X=(\xi_1, \dots, \xi_n)$ is the last row. 

As an illustration, let us give a short proof for the classical bound $p_n =o(1)$ (first showed by Koml\'os in \cite{Kom1} using a different argument). 

\begin{theorem} \label{theorem:singular}  $p_n =o(1) $. \end{theorem} 

We with a simple observation  \cite{KKSz}. 

\begin{fact} Let $H$ be a subspace of dimension $1 \le d\le n$. Then $H$ contains at most $2^{d}$ Bernoulli vectors. \end{fact}

To see this, notice that in a subspace of dimension $d$, there is a set of $d$ coordinates which determine the others.  This fact implies

$$p_n  \le \sum_{i=1}^{n-1} \P(\Bx_{i+1} \in H_{i} ) \le \sum_{i=1}^{n-1}  2^{i-n} \le  1 -\frac{2}{2^{n}},$$

where $H_i$ is the subspace generated by the the first $i$ rows $\Bx_1,\dots,\Bx_i$ of $M_n$.
 
This bound is quite  the opposite of what we want to prove. However,  we notice that the  loss comes at the end. Thus, to obtain the desired upper bound $o(1)$, it suffices to show that the sum of 
the last 
(say) $\log \log n$ terms is at most  (say) $ \frac{1}{\log^{{1/3}} n}$. To do this, we will exploit the fact that the $H_{i}$ are spanned by random vectors.  The following lemma
(which is a more effective version of the above fact)  implies the theorem via the union bound. 

\begin{lemma} \label{lemma1} Let $H$ be the subspace spanned by $d$ random vectors, where $d \ge n-\log \log n$. Then with probability at least $1- \frac{1}{n}$, $H$ contains at most  $\frac{2^{n}}{\log^{{1/3}} n }$ Bernoulli  vectors. \end{lemma}

 We say that  a set $S$ of $d$ vectors is {\it $k$-universal}  if for any   set of $k$ different indices $1 \le i_{1}, \dots, i_{k} \le n$ and any set of signs $\ep_{1}, \dots, \ep_{n}$ ($\ep_{i}= \pm 1$), there is a vector $V$ in $S$ such that the sign of the $i_{j}$-th  coordinate of $V$ matches $\ep_{j}$, for all $1\le j \le k$.
 
 \begin{fact} If $d \ge n/2$, then with probability at least $1-\frac{1}{n}$, a set  of 
 $d$ random vectors is $k$-universal, for $k=  \log n/10$. \end{fact}
 
 To prove this, notice that the failure probability is, by the union bound, at most 
 
 $${n \choose k} (1 -\frac{1}{2^{k} })^{d} \le n^{k}(1 -\frac{1}{2^{k}} )^{n/2}  \le n^{{-1}}. $$

 If $S$ is $k$-universal, then 
  any non-zero vector $\Bv$ in the orthogonal complement of the subspace spanned by $S$ should have more than $k$ non-zero vectors (otherwise, there would be a vector in $S$ having positive inner product with $\Bv$). If we fix such $\Bv$, and let $\Bx$ be a random Bernoulli vector, then by Theorem \ref{theorem:Erdos1} 
  
  $$\P (\Bx \in \hbox{ span} (S)) \le \P(\Bx \cdot \Bv=0) = O(\frac{1}{k^{1/2}}) =o(  \frac{1}{\log ^{{1/3}} n}),$$
  
  \noindent proving Lemma \ref{lemma1} and Theorem \ref{theorem:singular}. 
  
  \vskip2mm 
  
  The symmetric version of Theorem \ref{theorem:singular} is much harder and has been open for quite sometime (the problem was raised by  Weiss  the 1980s). Let $p^{sym}_n$ be the singular probability of a random symmetric matrix whose upper diagonal entries are iid Bernoulli variables. Weiss conjectured that 
  $p^{sym}_n =o(1)$. This was proved by Costello, Tao, and the second author \cite{CTV}. Somewhat interestingly, this proof made use of 
  the argument of Koml\'os in \cite{Kom1} which he applied for non-symmetric matrices. Instead of exposing the matrix row by row, one needs to 
  expose the principal  minors one by one, starting with the 
  top left entry. At step $i$, one has a symmetric matrix $M_i$ of size $i$ and the next matrix $M_{i+1}$  is obtained by adding a row and its transpose. Following Koml\'os,  one  defines 
   $X_i$ as the co-rank of the matrix at step $i$ and shows that the sequence $X_i$ behaves as a bias  random walk with a positive drift.  Carrying out the calculation carefully, 
   one obtains that $X_n=0$ with high probability. 
  
  The key technical step of this argument is to show that if $M_i$ has full rank than so does $M_{i+1}$, with very high probability. Here the 
   quadratic Littlewood-Offord bound is essential. Notice that if  we condition on
    $M_i$, then $\det(M_{i+1})$ is a quadratic form of the entries in the additional  ($(i+1)$-th) row, with coefficients 
   being the co-factors of $M_i$.  By looking at these 
   co-factors closely and using Theorem \ref{theorem:LO:CTV} (to be more precise, a variant of it where only a fraction of coefficients are required to be non-zero), 
   one can establish Weiss' conjecture.

  \begin{theorem} \label{theorem:psym} 
  $$p_n^{sym} = o(1). $$ 
   \end{theorem}

Getting  strong quantitative bounds for $p_n$ and $p_n^{sym}$ is more challenging, and we  will continue 
this topic in Sections \ref{section:singularity2} and \ref{section:singularity3}, after the introduction of inverse theorems.

\section{Hal\'asz' results}  \label{section:Halasz} 

In \cite{H} (see also in \cite{TVbook}),  Hal\'asz
proved the following very general theorem. 

\begin{theorem} \label{theorem:Hal} Suppose that there exists a constant $\delta>0$ such that the following holds

\begin{itemize}
\item (General position) for any unit vector $\Be$ in $\R^d$ one can select at least $\delta n$ vectors $a_k$ with $|\langle a_k,\Be \rangle|\ge 1$;
\vskip .1in
\item (Separation) among the $n^d$ vectors $b$ of the form $\pm a_{k_1}\pm \dots \pm a_{k_d}$ one can select at least $\delta n^d$ with pairwise distance at least 1.
\end{itemize} 

Then 

$$\rho_{d, 1,\Ber}(A)=O_{\delta,d}(n^{-3d/2}).$$
\end{theorem}

Hal\'asz'  method is Fourier analytic, which uses the following powerful  Ess\'een-type concentration inequality as the starting point (see \cite{H},\cite{Esseen}).

\begin{lemma} \label{lemma:Esseen} There exists an absolute positive constant $C=C(d)$ such that for any random variable $X$ and any unit ball $\BB \subset \R^d$

\begin{equation}\label{eqn:Esseen} 
\P( X \in \BB ) \leq C \int_{\|t\|_2 \leq 1} |\E( \exp({i \langle t,X\rangle }))|\ dt.
\end{equation}
  \end{lemma} 

\begin{proof}(of Lemma \ref{lemma:Esseen}) With the function $k(t)$ to be defined later, let $K(x)$ be its Fourier's transform

$$K(x)= \int_{\R^d} \exp(i \langle x,t \rangle ) k(t)dt.$$

Let $H(x)$ be the distribution function and $h(x)$ be the characteristic function of $X$ respectively. By Parseval's indentity we have 

\begin{equation}\label{eqn:Parseval}
\int_{\R^d} K(x) dH(x)= \int_{\R^d} k(t)h(t)dt.
\end{equation}

If we choose $k(t)$ so that 

\[
\begin{cases}
k(t)=0 \mbox{ for } \|t\|_2\ge 1,\\ 
|k(t)|\le c_1 \mbox{ for } \|t\|_2\le 1,
\end{cases}
\]

then the RHS of \eqref{eqn:Parseval} is bounded by that of \eqref{eqn:Esseen} modulo a constant factor. 

Also, if 

\[
\begin{cases}
K(x)\ge 1, \|x\|_2\le c_2, \mbox{ for some constant } c_2, \\
K(x)\ge 0 \mbox{ for } \|x\|_2\ge c_2,$$
\end{cases}
\]

then the LHS of \eqref{eqn:Parseval} is at least $\int_{\|x\|_2\le c_2}dH(x)$. 

Similarly, by translating $K(x)$ (i.e. by multiplying $k(x)$ with a phase of $\exp(i\langle t_0,x\rangle$), we obtain the same upper bound for  $\int_{\|x-t_0\|_2\le c_2}dH(x)$. Thus, by covering the unit ball $\BB$ with balls of radius $c_2$, we arrive at \eqref{eqn:Esseen} for some constant $C$ depending on $d$.

To construct $k(t)$ with the properties above, one may take it to have the convolution form 

$$k(x):=\int_{x\in \R^d} k_1(x)k_1(t-x)\\dx,$$ 

where $k_1(x)=1$ if $\|x\|_2\le 1/2$ and $k_1(x)=0$ otherwise.

 \end{proof}

To illustrate Hal\'asz' method, let us give a quick proof of Erd\H os bound $O(n^{-1/2})$ for 
the small ball probability $\rho_{1,1,\Ber} (A)$ with $A$ being  a multi-set of $n$ real numbers of absolute value at least 1. 
In view of Lemma \ref{lemma:Esseen}, it suffices to show that 

$$ \int_{|t| \leq 1} |\E( \exp( it \sum_{j=1}^n a_j \xi_j )| )\ dt = O( 1 / \sqrt{n} ).$$

By the independence of the $\xi_j$, we have

$$ |\E( \exp( it \sum_{j=1}^n a_j \xi_j ) )|  =
\prod_{j=1}^n  |\E( \exp(i t a_j \xi_j)| = |\prod_{j=1}^n \cos(t
a_j)|. $$ 

By H\"older's inequality
$$ \int_{|t| \leq 1}| \E( \exp( it \sum_{j=1}^n a_j \xi_j ))|\ dt
\leq \prod_{j=1}^n (\int_{|t| \leq 1} |\cos(t a_j)|^n \
dt)^{1/n}.$$ But since each $a_j$ has magnitude at least 1, it is
easy to check that $\int_{|t| \leq 1} |\cos(t a_j)|^n \ dt = O( 1 /
\sqrt{n} )$, and the claim follows.

Using Hal\'asz technique, it is possible to deduce 

\begin{corollary}\cite[Corollary 7.16]{TVbook} \label{cor:Hal} Let $A$ be a multi-set in $\R$. 
Let $l$ be a fixed integer and $R_l$ be the number of solutions of the equation
$a_{i_1}+\dots+a_{i_l}=a_{j_1}+\dots+a_{j_l}$. Then

$$\rho_A:= \sup_{x} \P(S_A=x) = O(n^{-2l - \frac{1}{2}}R_l).$$

\end{corollary} 

This result provides the  hierarchy of bounds mentioned in the previous section,  given that we forbid more and more additive structures on $A$. Let us consider the first few steps of the hierarchy.

\begin{itemize} 

\item If the $a_i$'s are distinct, then we can set $l=1$ and $R_1=n$ (the only solutions are the trivial ones $a_i=a_i$). Thus, we obtain S\'ark\"ozy-Szemer\'edi's  bound 
$O(n^{-3/2})$.  

\vskip .1in

\item If we forbid the $a_i$'s to satisfy equations $a_i +a_j = a_l +a_k$, for any $\{i,j\} \neq \{k,l\}$ (in particular this prohibits $A$ to be an arithmetic progression), then one can fix $l=2$ and $R_2 = n^2$
and obtain $\rho_A = O (n^{-5/2}). $

\vskip .1in

\item If we continue to forbid equations of the form $a_h+ a_i +a_j =a_k +a_l +a_m $, $\{h,i,j\}\neq \{k,l,m\}$, then one obtains $\rho_A=O(n^{-7/2})$ and so on. 

\end{itemize}

Hal\'asz' method is very powerful and has a strong influence on the recent developments discussed  in 
the coming sections.

\section{Inverse theorems: Discrete case}  \label{section:discrete}

A few years ago,  Tao and the second author \cite{TVinverse} brought a new view to the small ball  problem. Instead of working out a hierarchy of bounds by 
 imposing new assumptions as done in Corollary \ref{cor:Hal}, they tried  to find  the underlying reason as to why the small ball probability is large (say, polynomial in $n$). 

\vskip .1in

It is easier and more natural to work with the discrete problem first. Let $A$ be a  multi-set of integers and $\xi$ be the Bernoulli random variable. 

\begin{question}[Inverse problem, \cite{TVinverse}] 
Let $n \rightarrow \infty$. Assume that for some constant $C$ 

$$\rho_A = \sup_x \P (S_A  =x ) \ge n^{-C}.$$ 

What can we say about the elements  $a_1,\dots,a_n$  of $A$ ? 
\end{question}

Denote by $M$ the sum of all elements of $A$ and rewrite  $\sum_{i} a_i \xi_i$  as $M -2 \sum_{i; \xi_i =-1} a_i $. As  $A$ has $2^n$ subsets, the bound $\rho_A  \ge n^{-C}$ implies that at least $2^n /n^C$ among the subset sums are exactly  $(M-x)/2$. This overwhelming collision suggests that $A$ must have some strong additive structure. 
Tao and the second author proposed 

{\bf Inverse Principle}:
\begin{equation} \label{IP}
\hbox { \it A set  with large small ball probability must have strong additive structure.}  
\end{equation}

The issue is, of course, to quantify the statement. 
Before attacking this question, let us recall the famous Freiman's inverse theorem from  Additive Combinatorics. 
As the readers will see, this theorem strongly motivates our study.

In the 1970s,  Freiman considered the collection of pairwise sums $A+A := \{ a+a' |a, a'  \in A \} $ \cite{Freimanbook}. Normally, one expects 
this collection to have $\Theta (|A|^2) $ elements. 
Freiman proved a deep and powerful theorem showing that if 
$A+A$  has only $O(|A|)$ elements (i.e, a huge number of collision occurs) then $A$ must look like an arithmetic progression. (Notice that if $A$ is an arithmetic progression then $|A+A| \approx 2|A|$.) 

To make Freiman's statement more precise, we need the definition of  \emph{generalized arithmetic progressions} (GAPs). 

\vskip .1in

\begin{definition}
A set $Q$ is a \emph{GAP of rank $r$} if it can be expressed  in the form
$$Q= \{g_0+ m_1g_1 + \dots +m_r g_r| M_i \le m_i \le M_i', m_i\in \Z \hbox{ for all } 1 \leq i \leq r\}$$ for some $g_0,\ldots,g_r,M_1,\ldots,M_r,M'_1,\ldots,M'_r$. \end{definition} 

It is convenient to think of $Q$ as the image of an integer box $B:= \{(m_1, \dots, m_r) \in \Z^r| M_i \le m_i \le M_i' \} $ under the linear map
$$\Phi: (m_1,\dots, m_r) \mapsto g_0+ m_1g_1 + \dots + m_r g_r. $$
The numbers $g_i$ are the \emph{generators } of $P$, the numbers $M_i',M_i$ are the \emph{dimensions} of $P$, and $\Vol(Q) := |B|$ is the \emph{volume} of $B$. We say that $Q$ is \emph{proper} if this map is one to one, or equivalently if $|Q| = \Vol(Q)$.  For non-proper GAPs, we of course have $|Q| < \Vol(Q)$.
If $-M_i=M_i'$ for all $i\ge 1$ and $g_0=0$, we say that $Q$ is {\it symmetric}.

If $Q$ is symmetric and $t >0$, the dilate $tQ$ is the set

$$  \{ m_1g_1 + \dots +m_r g_r| - tM_i' \le m_i \le  tM'_i \hbox{ for all } 1 \leq i \leq r\}.$$ 

It is easy to see that if $Q$ is a proper map of rank $r$, then $|Q+Q| \le 2^r |Q|$. This implies that if $A$ is a subset of density $\delta$ in a proper GAP Q of rank $r$, then as far as $\delta =\Theta (1)$,

$$|A +A| \le |Q+Q| \le 2^r |Q| \le \frac{2^r} {\delta} |A| =O(|A| ). $$   

Thus, dense subsets of a proper GAP of constant rank satisfies the assumption $|A+A| =O(|A|)$.  Freiman's remarkable  inverse theorem showed that this example 
is the only one. 

\begin{theorem}[Freiman's inverse theorem in $\Z$]\label{theorem:Freiman} Let
  $\gamma$ be a given positive number. Let $X$ be a set in $\Z$ such that $|X+X| \le
\gamma |X|$. Then  there exists a proper
GAP of rank $O_\gamma(1)$ and cardinality
$O_{\gamma}(|X|)$ that contains  $X$.
\end{theorem}

For further discussions, including a beautiful proof by Ruzsa, see \cite[Chapter 5]{TVbook}; see also  \cite{BGT} for 
recent and deep developments concerning non-cummutative settings (when $A$ is a subset of a non-abelian group).

In our case, we want to find examples for $A$ such that $\rho(A) : =\sup_{x} \P(S_A =x)$ is large. 
Again, dense subsets of a proper GAP come in as natural candidates. 

\begin{example} \label{example:linear:1} Let $Q$ be a proper symmetric GAP of rank $r$ and volume $N$. Let $a_1, \dots, a_n$ be (not necessarily distinct) elements of $Q$. By the Central Limit Theorem, with probability at least $2/3$, 
the random sum 
$S_A =\sum_{i=1}^n a_ix_i$ takes value in the dilate  $10 n^{1/2} Q$. Since $|tQ| \le t^r N$, by the pigeon hole principle, we can conclude that  there is a point $x$ where 

$$\P(S_A=x )= \Omega (\frac{1}{n^{r/2} N}).$$ 

Thus if  $|Q|= N=O(n^{C-r/2})$ for some constant $C\ge r/2$, then 

$$\rho(A) \ge \P(S_A =x) = \Omega (\frac{1}{n^{C}}).$$

\end{example}

This  example  shows that  if the elements of $A$ are  elements of a symmetric proper GAP with a small rank and small cardinality, then
$\rho(A)$ is large. Inspired by Freiman's theorem,  Tao and the second author  \cite{TVstrong,TVinverse} showed  that the converse is also true.

\begin{theorem} \label{theorem:TV1} For any constant $C, \epsilon$ there  are constants $r, B$ such that the following holds. 
 Let $A$ be a multi-set of $n$  real numbers such that 
$\rho(A) \ge n^{-C}$, then there is a GAP $Q$ of  rank $r$   and volume $n^B$ such that all but $n^{\epsilon}$ elements of 
$A$ belong to $Q$. 
\end{theorem} 

The dependence of $B$ on $C, \epsilon$ is not explicit in \cite{TVinverse}. In \cite{TVstrong}, Tao and the second author obtained an almost sharp dependence. 
The best dependence, which mirrors Example \ref{example:linear:1} was proved in a more recent paper  \cite{NgV} of the current authors. This proof is 
different from those in earlier proofs and made a direct use of Freiman's theorem (see Appendix  \ref{section:discrete:proof}).

\begin{theorem}[Optimal inverse Littlewood-Offord theorem, discrete case]\cite{NgV}\label{theorem:ILO:optimal}
Let $\eps<1$ and $C$ be positive constants. Assume that
$$\rho (A)  \ge  n^{-C}. $$ Then there exists a proper symmetric GAP $Q$ of rank $r= O_{C, \eps} (1)$ which contains all but at most $\eps n$
 elements of $A$ (counting multiplicities), where
 $$|Q| = O_{C, \eps} ( \rho (A)^{-1} n^{- \frac{r}{2}}). $$
\end{theorem}

The existence of the exceptional set cannot be avoided completely. For more discussions, see \cite{TVinverse, NgV}.  There is also a trade-off between the size of the exceptional set and the bound on $|Q|$. In many combinatorial applications (see, for instance, the next section),  an exceptional set of size $\epsilon n $ 
does not create any trouble. 

Let us also point out that the above inverse theorems hold in a  very general setting  where the random variables $\xi_i$ are not necessarily 
Bernoulli and independent  (see \cite{TVinverse, TVstrong, NgV,Ng-zerosum} for more details).


\section{Application: From Inverse to Forward}  \label{section:ItoF}

 One can use  the "inverse" Theorem \ref{theorem:ILO:optimal} to quickly prove  several "forward" theorems presented in earlier sections. 
 As an example, let us derive Theorems \ref{theorem:Erdos1} and \ref{theorem:SSz1}. 

\begin{proof}(of Theorem \ref{theorem:Erdos1})
Assume, for contradiction, that there is a set $A$ of $n$ non-zero numbers such that $\rho(A) \ge c_1 n^{-1/2}$ for some large constant $c_1$ to be chosen. Set $\eps=.1, C=1/2$.
By Theorem \ref{theorem:ILO:optimal}, there is a GAP $Q$ of rank $r$ and size $O (\frac{1}{c_1} n^{C-\frac{r}{2}} )$ that contains
at least $.9n$ elements from $A$.  However, by setting $c_1$ to be sufficiently large (compared to the constant in big O)  and using the fact that $C=1/2$ and $r \ge 1$, 
we can force  $O (\frac{1}{c_1} n^{C-\frac{r}{2}} ) < 1$. Thus, $Q$ has   to be empty, a contradiction.
\end{proof}

\begin{proof}(of Theorem \ref{theorem:SSz1}) Similarly, assume that there is a set $A$ of $n$ distinct numbers such that $\rho(A) \ge c_1 n^{-3/2}$ for some large constant $c_1$ to be chosen. Set $\eps=.1, C=3/2$.
By Theorem \ref{theorem:ILO:optimal}, there is a GAP $Q$ of rank $r$ and size $O(\frac{1}{c_1} n^{C-\frac{r}{2}} )$ that contains
at least $.9n$ elements from $A$. This implies $|Q| \ge .9 n$. By setting $c_1$ to be sufficiently large and using the fact that   $C=3/2$ and $r \ge 1$, we can guarantee that $|Q|  \le .8n $, a contradiction.
\end{proof}

The readers are invited to work out the proof of Corollary \ref{cor:Hal}. 

Let us now  consider another application of Theorem \ref{theorem:ILO:optimal}, which enables us to make very precise counting arguments. Assume that we would like to count the  number of multi-sets $A$ of integers with $\max |a_{i}| \le M=n^{O(1)}$ such that $\rho(A) \ge n^{-C}$.

Fix $d \ge 1$, fix \footnote{A more detailed version of Theorem \ref{theorem:ILO:optimal} tells us that there are
 not too many ways to choose the
 generators of $Q$. In particular, if $|a_i|\le M=n^{O(1)}$, the number of ways to fix these is negligible compared to the main term.} a  GAP $Q$  with rank $r$ and
 volume $|Q|\le c \rho(A)^{-1}n^{-\frac{r}{2}}$ for some constant $c$ depending on $C$ and $\ep$. The dominating term in the calculation will be the number of multi-sets which intersect with $Q$
 in subsets of size at least $(1-\ep)n$. This number is bounded by

\begin{align}\label{discretcounting}  
\sum_{k\le \ep n}|Q|^{n-k} (2M)^{k}&\le  \sum_{k\le \ep n}(c\rho(A)^{-1}n^{-\frac{r}{2}})^{n-k} (2M)^k\\\nonumber 
&\le (O_{C,\ep}(1))^n n^{O_\ep(1)n} \rho(A)^{-n} n^{-\frac{n}{2} }.
\end{align}

We thus obtain the following useful result.

\begin{theorem} [Counting theorem: Discrete case] \label{counting} 
The number $N$ of multi-sets $A$ of integers with $\max |a_{i}| \le n^{C_1}$ and $\rho(A) \ge n^{-C_2}$ is bounded by 

 $$N =\big(O_{C_1,C_2,\ep}(1)\big)^n n^{O_\ep(1)n} \big(\rho(A)^{-1} n^{-1/2}\big)^n,$$

where $\ep$ is an arbitrary constant between $0$ and $1$.
\end{theorem} 

Due to their asymptotic nature, our  inverse theorems  do not directly imply Stanley's precise result (Theorem \ref{theorem:Stan}).
 However, by refining the proofs, one can actually get very close and with some bonus, namely,  additional strong  {\it rigidity} information. For instance, in  \cite{Ng-Stanley} the first author showed that if the elements of $A$ are distinct, then 

$$\P (S_A =x)  \le (\sqrt {\frac{24}{\pi}}+o(1)) n^{-3/2},$$

\noindent where the constant on the RHS is obtained when  $A$ is 
 the symmetric arithmetic progression $A_0$ from Theorem \ref{theorem:Stan}. 
It was showed that if $\rho(A)$ is close to this value, then $A$ needs to be very close to a symmetric arithmetic progression. 

\begin{theorem}\cite{Ng-Stanley}\label{theorem:stable} There exists a positive constant $\epsilon_0$ such that for any $0 < \epsilon \le \epsilon_0$, there exists a positive number $\epsilon' = \epsilon'(\epsilon)$ such that $\ep'\rightarrow 0$ as $\ep \rightarrow 0$ and the following holds: if $A$ is a set of $n$ distinct integers and

    $$\rho(A)\ge \Big(\sqrt {\frac{24}{\pi}} -\ep\Big) n^{-\frac{3}{2}},$$
    then there exists an integer $l$ which divides all $a\in A$ and $$\sum_{a\in A}\Big(\frac{a}{l}\Big)^2 \le (1+\ep')\sum_{a\in A_0}a^2.$$
\end{theorem}

We remark  that a slightly weaker stability can be shown even when we have a much weaker assumption $\rho(A)\ge \ep n^{3/2}$. 

As the reader will see, in many applications in the following sections, we do not use the inverse theorems directly, but rather  their counting corollaries, such as  
Theorem \ref{counting}. Such counting results can be used to bound the probability of a bad event through the union bound 
(they count the number of terms in the union).  This method was first used in studies of random matrices \cite{TVsing, TVinverse, RV}, but it is simpler  to 
illustrate the idea  by the following  more recent result of Conlon, Fox, and Sudakov \cite{CFS}. 

A Hilbert cube  is a set of the form $x_0 + \Sigma (\{x_1, \dots, x_d \}) $
where $\Sigma (X) =\{ \sum_{x \in Y }  x | Y \subset X\} $, and $0 \le x_0, 0< x_1 < \dots < x_d$ are integers. 
 Following the literature, we  refer to the index $d$ as the dimension. 
One of the earliest results in Ramsey theory is a theorem of Hilbert \cite{Hil} stating that  for any fixed $r$ and $d$ and  $n$  sufficiently 
large, any coloring of the set $[n] := \{1, \dots, n \}$ 
with $r$  colors must contain a monochromatic 
Hilbert cube of dimension $d$. Let $ h(d, r)$ be the smallest such $n$.  The best known upper bound 
for this function is  \cite{Hil, GRS} 

$$ h(d, r)  \le (2r)^{2^{d-1}} . $$

The density version of \cite{Szem}  states that for any natural number $d$ and $\delta >0$  there exists an $n_0$ such that if 
$n \ge n_0$ then  any subset of $n$ of density $\delta$ contains a Hilbert cube of dimension $d$.  
One can show that 

$$d \ge c \log \log n $$ where $c$ is a positive constant depending only on $\delta$.  

On the other hand, 
Hegyv\'ari shows an upper bound of the form $O( \sqrt {\log n \log \log n })$ by considering a random subset of density $\delta$. 
Using the discrete inverse theorems (Section \ref{section:discrete}), Conlon, Fox, and Sudakov \cite{CFS} removed the $\log \log n$ term, obtaining 
$O(\sqrt {\log n})$, which is sharp up to the constant in big $O$, thanks to another result of Hegyv\'ari. 

Conlon et. al. started with the following corollary of Theorem \ref{theorem:TV1}.

\begin{lemma}  For every $C > 0,  1 > \epsilon >0$  there exist positive constants $r$  and $C'$  such that if  $X$
is a multiset with $d$  elements and $|\Sigma (X)| \le d^C$, then 
 there is a GAP $Q $ of dimension $ r $ and volume at 
most  $d^{C'} $ such that all  but at most $d^{1-\epsilon}$ elements of $X$  are contained in $Q$. 
\end{lemma}

From this, one can easily prove the following counting lemma. 

\begin{lemma}  For  $s \le \log d$, the number of $d$-sets $X \subset  [n]$  with $\Sigma (X ) \le 2^s d^2 $  is at most $n^{O(s)} d^{O(d)}$ . 
\end{lemma} 

Let $A$ be a random set of $[n]$ obtained by choosing each number with probability $\delta $ independently. Let $E$ be the event that 
$A$ contains a Hilbert cube of dimension $c \sqrt {\log n }$. We aim to show that 

\begin{equation} \P(E) =o(1) , \end{equation}  given $c$ sufficiently large.

Trivially $\P (E) \le n \sum_{X \subset [n]} \delta ^{|\Sigma (X)| }, $ where the factor $n$ corresponds to the number of ways to choose $x_0$. Let $m_t$ be the number of $X$ such that 
$|\Sigma(X)| =t$.  The RHS can be bounded from above by $n \sum _t m_t \delta ^t $. 

If $t$ is large, say $t  \ge d^3$, we just crudely bound $\sum_{t \ge d^3} m_t $ by $n^d$ (which is the total number of ways to choose $x_1, \dots, x_d$). 
The contribution in probability in this case is at most $n \times n^d \times \delta ^{d^3} =o(1) $, if $c$ is sufficiently large. In the case $t  < d^3$, we make use of the counting lemma above to bound $m_t$ and a routine calculation finishes the job.

\section{Inverse Theorems: Continuous case I.}  \label{section:cont1} 

In this section and the next, we consider sets with large small probability. 

We say that a vector $v \in   \R^d$ is {\it $\delta$-close} to a set $Q\subset \R^d$ if there exists a vector $q\in Q$ such that $\|v-q\|_2 \le \delta$.
A set $X$ is  $\delta$-close to a set $Q$ if every element of $X$ is $\delta$-close to $Q$.
The continuous analogue of Example \ref{example:linear:1} is the following.

\begin{example} \label{example:linear:2}
Let $Q$ be a proper  symmetric GAP of rank $r$ and volume $N$ in $\R^d$.
Let $a_1, \dots, a_n$ be (not necessarily distinct) vectors  which are
$\frac{1}{100} \beta  n^{-1/2}$-close to $Q$. Again by the Central Limit Theorem, 
with probability at least $2/3$, $S_A$ is $\beta$-close to $10 n^{1/2} Q$. Thus, by the pigeon hole principle, there is a 
point $x$ in $100 n^{1/2} Q$ such that 

$$\P (S_A \in B(x, \beta)) \ge |10 n^{1/2} Q|^{-1} \ge \Omega (n^{-r/2} |Q|^{-1}  ). $$

 It follows that if  $Q$ has cardinality $n^{C-\frac{r}{2}}$ for some constant $C\ge r/2$, then

\begin{equation} \label{bound3} \rho_{d, \beta, \Ber} (A)   = \Omega (\frac{1}{n^{C}}).
\end{equation}
\end{example}

Thus,  in view of  the Inverse Principle \eqref{IP} and Theorem \ref{theorem:ILO:optimal} , we 
 would expect that if $\rho_{d,\beta, \Ber} (A) $ is large, then most of the $a_i$ is close to a GAP with small volume. 
 This statement turned out to hold for very general random variable $\xi$ (not only for Bernoulli).  In practice, we  can consider any 
 real random variable  $\xi$, which satisfies the following condition: there are positive constants $C_1,C_2,C_3$ such that

\begin{equation}\label{eqn:2ndmoment}
\P(C_1\le |\xi_1-\xi_2| \le C_2)\ge C_3,
\end{equation}

\noindent where $\xi_1,\xi_2$ are iid copies of $\xi$. 

\begin{theorem}\cite{NgV} \label{theorem:ILO:continuous} Let $\xi$ be a real  random variable satisfying \eqref{eqn:2ndmoment}. 
Let $0 <\ep < 1; 0 < C$ be constants and  $ \beta >0$ be a parameter that may depend on $n$. Suppose that $A=\{a_1,\dots,a_n\}$ is a (multi-)subset of $\R^d$ such that $\sum_{i=1}^n\|a_i\|_2^2=1$ and that $A$ has large small ball probability
$$\rho:= \rho_{d, \beta,\xi}(A)\ge n^{-C}. $$ Then there exists a symmetric proper GAP $Q$ of constant rank $r \ge d$ and of size 
$|Q| =O(\rho^{-1} n^{(-r+d)/2})$ such that all but $\epsilon n$ elements of $A$ are   are $O(\frac{\beta \log n }{n^{1/2} })$-close to $Q$.

\end{theorem}

The next result gives more information about  $Q$, but with a weaker approximation.

\begin{theorem}\label{corollary:continuous:NgV} Under the assumption of the above theorem, the following holds.
 For  any number $n'$ between $n^\ep$ and $n$, there exists a proper symmetric GAP $Q=\{\sum_{i=1}^r x_ig_i : |x_i|\le L_i \}$ such that

\begin{itemize}

\item At least $n-n'$ elements of $A$  are $\beta$-close to $Q$.

\vskip .1in

\item $Q$ has small rank, $r=O (1)$, and small cardinality

$$|Q| \le \max \left(O(\frac{\rho^{-1}}{\sqrt{n'}}),1\right).$$

\vskip .1in

\item There is a non-zero integer $p=O(\sqrt{n'})$ such that  all
 steps $g_i$ of $Q$ have the form  $g_i=(g_{i1},\dots,g_{id})$, where $g_{ij}=\beta \frac{p_{ij}} {p} $ with  $p_{ij} \in \Z$ and $p_{ij}=O(\beta^{-1}\sqrt{n'}).$

\end{itemize}
\end{theorem}

Theorem \ref{corollary:continuous:NgV} immediately implies the following result  which can be seen as a continuous  analogue of  Theorem \ref{counting}. This result was first proved by Tao and the second author for the purpose of verifying the Circular Law in random matrix theory \cite{TVcir,TVinverse} using  a more complicated argument.

Let $n$ be a positive integer and $\beta,\rho$ be positive numbers that may depend on $n$. Let $\mathcal{S}_{n,\beta, \rho}$ be the collection of all multisets $A=\{a_1,\dots,a_n\} , a_i \in \R^2 $ such that  $\sum_{i=1}^n \| a_i\|_2 ^2=1$ and $\rho_{d,\beta,\Ber}(A)\ge \rho$.

\vskip .2in

\begin{theorem}[Counting theorem, continuous case] \cite{TVcir,TVinverse}\label{theorem:betanet}
Let  $0 <\ep \le 1/3$ and $C>0$ be constants. Then, for all sufficiently large $n$  and $\beta \ge \exp(-n^{\ep})$
and $\rho \ge n^{-C}$ there is a set $\mathcal{S}\subset (\R^2)^n$ of size at most

$$\rho^{-n}n^{-n(\frac{1}{2}-\epsilon) }
  + \exp(o(n))$$

\noindent such that for any $A=\{a_1,\dots,a_n\} \in \mathcal{S}_{n,\beta,\rho}$ there is some $A'=(a_1',\dots, a_n')\in \mathcal{S}$ such that $\|a_i-a_i'\|_{2}  \le \beta$ for all $i$.
\end{theorem}

\begin{proof}(of Theorem \ref{theorem:betanet}) Set $n':=n^{1-\frac{3\ep}{2}}$ (which is $\gg n^{\ep}$ as $\ep\le 1/3$). Let $\mathcal{S'}$ be the collection of
 all subsets of size at least $n-n'$ of GAPs whose parameters satisfy the conclusion of Theorem \ref{corollary:continuous:NgV}.

Since each GAP is determined by its generators and dimensions, the number of such GAPs is bounded by $((\beta^{-1}\sqrt{n'})\sqrt{n'})^{O(1)}
(\frac{\rho^{-1}}{\sqrt{n'}})^{O(1)}=\exp(o(n))$.  (The term $(\frac{\rho^{-1}}{\sqrt{n'}})^{O(1)}$
bounds the number of choices of the dimensions $M_i$.)  Thus 

$$|\mathcal{S'}|= \left(O((\frac{\rho^{-1}} {\sqrt{n'}})^{n})+1\right) \exp(o(n)).$$

We approximate each of the exceptional elements by a lattice point in $\beta\cdot (\Z/d)^d$. Thus if we let $\mathcal{S''}$ to be the set of these approximated tuples then $|\mathcal{S''}|\le \sum_{i\le n'} (O(\beta^{-1}))^i = \exp(o(n))$ (here we used the assumption $\beta \ge \exp(-n^{\ep})$).

Set $\mathcal{S}:=\mathcal{S'}\times \mathcal{S''}$. It is easy to see that
 $|\mathcal{S}|\le O(n^{-1/2+\ep}\rho^{-1})^n+\exp(o(n))$. Furthermore, if $\rho(A)\ge n^{-O(1)}$ then $A$ is $\beta$-close to an element of $\mathcal{S}$, concluding the proof.
\end{proof}

\section{Inverse theorems: Continuous case  II.}\label{section:cont2}

Another realization of the Inverse Principle \eqref{IP} was given by Rudelson and Vershynin  in \cite{RV,Versur} (see also Friedland and Sodin \cite{FS}).   
 Let $a_1,\dots,a_n$ be real numbers. Rudelson and Vershynin defined the 
  essential {\em least common denominator} ($\LCD$) of $\Ba=(a_1,\dots,a_n)$  as follows. Fix parameters $\alpha$  and $\gamma$, where $\gamma \in (0,1)$, and define
$$
\LCD_{\alpha,\gamma}(\Ba)
:= \inf \Big\{ \theta > 0: \dist(\theta \Ba, \Z^n) < \min (\gamma \| \theta \Ba \|_2,\alpha) \Big\}.
$$

The requirement that the distance is smaller than $\gamma\|\theta \Ba\|_2$ forces us to consider only non-trivial integer points as approximations of $\theta \Ba$. One typically assume $\gamma$ to be a small constant, and $\alpha = c \sqrt{n}$ with a small constant $c>0$. The inequality $\dist(\theta \Ba, \Z^n) < \alpha$ then yields that most coordinates of $\theta \Ba$ are within a small distance from 
non-zero integers.

\begin{theorem}[Diophatine approximation]\cite{RV,RV-rec}\label{theorem:RV}
  Consider a sequence $A=\{a_1,\ldots,a_n\}$ of real numbers which satisfies
$\sum_{i=1}^n a_i^2 \ge 1$. Let $\xi$ be a random variable such that
$\sup_{a}\P(\xi \in B(a,1)) \le 1-b$ for some $b > 0$, and $x_1, \ldots, x_n$ be iid copies of $\xi$. Then, for every $\alpha > 0$ and $\gamma \in (0,1)$, and for
  $$
  \beta \ge \frac{1}{\LCD_{\alpha,\gamma}(\Ba)},
  $$
  we have
  $$
  \rho_{1,\beta,\xi}(A) \le \frac{C\beta}{\gamma \sqrt{b}} + Ce^{-2b\alpha^2}.
  $$
\end{theorem}

One can use Theorem \ref{theorem:RV} to prove a special case of  the forward result of Erd\H{o}s and Littlewood-Offord when  most of the $a_i$ have the same order of magnitude (see \cite[p. 6]{RV}). 
\footnote{One can also handle this case by conditioning on the abnormal $a_i$ and use  Berry-Esseen for the remaining sum.}   Indeed, assume that $K_1\le |a_i| \le K_2$ for all $i$, where $K_2=c K_1$ with $c=O(1)$. Set $a_i':=a_i/\sqrt{\sum_j a_j^2}$ and  $\Ba':=(a_1',\dots,a_n')$. Choose $\gamma=c_1, \alpha=c_2 \sqrt{n}$ with sufficiently small positive constants $c_1,c_2$ (depending on $c$), the condition $\dist(\theta \Ba', \Z^n) < \min (\gamma \| \theta \Ba' \|_2,\alpha)$ implies that $|\theta a_i'-n_i|\le 1/3$ with $n_i\in \Z,n_i\neq 0$ for at least $c_3n$ indices $i$, where $c_3$ is a positive constant depending on $c_1,c_2$ . It then follows that for these indices, $\theta^2{a_i'}^2  \ge 4n_i^2/9$. Summing over $i$, we obtain $\theta^2=\Omega(n)$ and so $\LCD_{\alpha,\gamma}(\Ba')=\Omega(\sqrt{n})$. Applying Theorem \ref{theorem:RV} to the vector $\Ba'$ with $\beta=1/\LCD_{\alpha,\gamma}(\Ba')$, we obtain the desired upper bound $O(1/\sqrt{n})$ for the concentration probability.


Theorems \ref{theorem:RV} is not exactly {\it inverse} in the Freiman sense. On the other hand, it is convenient to use and in most applications provides a  sufficient 
amount of structural information that allows one derive a counting theorem.  An extra advantage  here is that this theorem enables one 
to consider sets $A$ with  small ball probability as small as $(1-\epsilon)^n$, rather than just $n^{-C}$ as in Theorem \ref{theorem:ILO:continuous}.

The definition of the essential least common denominator above
can be extended naturally to higher dimensions. To this end, we define the product of such multi-vector $\Ba$ and a vector $\theta \in \R^d$ as

$$
\theta \cdot \Ba = (\langle \theta, a_1\rangle , \cdots, \langle \theta, a_n\rangle ) \in \R^n.
$$

Then we define, for $\alpha > 0$ and $\gamma \in (0,1)$,
$$
\LCD_{\alpha,\gamma}(\Ba)
:= \inf \Big\{ \|\theta\|_2: \; \theta \in \R^d,
     \dist(\theta \cdot \Ba, \Z^N) < \min(\gamma\|\theta \cdot \Ba\|_2,\alpha) \Big\}.
$$

The following generalization of Theorem \ref{theorem:RV} gives a bound on the small ball probability for the random sum $\sum_{i=1}^n a_i x_i$ in terms of the additive structure of the coefficient sequence $\Ba$.

\begin{theorem}[Diophatine approximation, multi-dimensional case]\cite{RV-rec,FS}\label{SBP}
  Consider a sequence $A=\{a_1,\ldots,a_n\}$ of vectors $a_i \in \R^d$,
  which satisfies
  \begin{equation}                                  \label{super-isotropy}
    \sum_{i=1}^n \langle a_i,x\rangle ^2 \ge \|x\|_2^2
    \qquad \text{for every $x \in \R^d$.}
  \end{equation}
  Let  $\xi$ be a random variable such that
  $\sup_{a}\P(\xi \in B(a,1)) \le 1-b$ for some $b > 0$ and $x_1, \ldots, x_n$ be iid copies of $\xi$.

  Then, for every $\alpha > 0$ and $\gamma \in (0,1)$, and for
  $$
  \beta \ge \frac{\sqrt{d}}{\LCD_{\alpha,\gamma}(\Ba)},
  $$
  we have
  $$
  \rho_{d,\beta \sqrt{d},\xi}(A) \le \Big( \frac{C\beta}{\gamma \sqrt{b}} \Big)^d + C^d e^{-2b\alpha^2}.
  $$
\end{theorem}

We will sketch the proof of Theorem \ref{theorem:RV} in Appendix \ref{section:proofcont2}.

\section{Inverse quadratic Littlewood-Offord}\label{section:quadratic2}

In this section, we revisit the quadratic Littlewood-Offord bound in Section \ref{section:quadratic1} and
consider its inverse.   We first consider a few examples of $A$ where (the quadratic small probability) $\rho_q(A)$ is large.

\begin{example}[Additive structure implies large small ball  probability]\label{example:quadratic:1} Let $Q$ be a proper symmetric GAP of rank $r=O(1)$ and of size $n^{O(1)}$. Assume that $a_{ij} \in Q$, then for any $\x_i\in \{\pm 1\}$

$$\sum_{i,j}a_{ij}\x_i\x_j\in n^2Q.$$ 

Thus, by the pegion-hole principle, 

$$\rho_q(A)\ge n^{-2r}|Q|^{-1} =n^{-O(1)}.$$ 
  
\end{example}

But unlike the linear case,  additive structure is not the only source for large small ball probability. Our next example shows that algebra also plays a role. 

\begin{example}[Algebraic structure implies large small ball probability]\label{example:quadratic:2}
Assume that 

$$a_{ij}= k_ib_j+k_jb_i$$ 

where $k_i\in \Z, |k_i|=n^{O(1)}$ and such that $\P(\sum_i k_i\x_i= 0)=n^{-O(1)}$. 

Then we have 

$$\P(\sum_{i,j}a_{ij}\x_i\x_j =0) =\P(\sum_{i}k_i\x_i  \sum_jb_j\x_j =0)=n^{-O(1)}.$$

\end{example}

Combining the above  two examples, we have the following general one. 

\begin{example}[Structure implies large small ball probability]\label{example:quadratic:3}
Assume that $a_{ij}=a_{ij}' +a_{ij}''$, where $a_{ij}'\in Q$, a proper symmetric GAP of rank $O(1)$ and size $n^{O(1)}$, and 

$$a_{ij}''= k_{i1}b_{1j}+k_{j1}b_{1i}+\dots+k_{ir}b_{rj}+k_{jr}b_{ri},$$ 

where $b_{1i},\dots, b_{ri}$ are arbitrary and $k_{i1},\dots,k_{ir}$ are integers bounded by $n^{O(1)}$, and $r=O(1)$ such that

$$\P\left(\sum_i k_{i1}\xi_i=0,\dots,\sum_i k_{ir}\xi_i=0\right)=n^{-O(1)}.$$ 

Then we have 

$$\sum_{i,j} a_{ij}\x_i\x_j = \sum_{i,j}a_{ij}'\x_i\x_j + 2(\sum_{i} k_{i1}\x_i)(\sum_{j} b_{1j}\x_j)+\dots + 2(\sum_{i} k_{ir}\x_i)(\sum_{j} b_{rj}\x_j).$$

Thus,

$$\P(\sum_{i,j}a_{ij}\x_i\x_j \in n^2Q)=n^{-O(1)}.$$

It then follows, by the pigeon-hole principle, that $\rho_q(A)=n^{-O(1)}$.
\end{example}

We have demonstrated the fact that as long as most of the $a_{ij}$ can be decomposed as 
 $a_{ij}=a_{ij}'+a_{ij}''$, where $a_{ij}'$ belongs to a GAP of rank $O(1)$ and size $n^{O(1)}$ and the symmetric matrix $(a_{ij}'')$ has rank $O(1)$,
  then $A=(a_{ij})$ has large quadratic small ball probability. The first author in \cite{Ng} showed that sort of the converse is also true.

\begin{theorem}[Inverse Littlewood-Offord theorem for quadratic forms]\label{theorem:ILO:quadratic} Let $\ep<1, C$ be positive constants. Assume that

$$\rho_q(A)\ge n^{-C}.$$ 

Then there exist index sets $I_0, I$ of size $O_{C,\ep}(1)$ and $n-O_C(n^\ep)$ respectively, with $I\cap I_0 = \emptyset$, and there exist integers 
$k_{ii_0}$ (for any pair $i_0\in I_0,i\in I$) of size bounded by $n^{O_{C,\ep}(1)}$,  and  a structured  set $Q$  of the form 
$$Q= \Big\{\sum_{h=1}^{O_C(1)} \frac{p_h}{q_h}   g_h|  p_h \in \Z, |p_h|, |q_h| =n^{O_{C,\ep}(1)}\Big\},$$
such that for all $i\in I$ the followings holds:

\begin{itemize}
 
\item (low rank decomposition)  for any $j\in I$,

$$a_{ij} = a_{ij}' - (\sum_{i_0\in I_0} k_{ii_0} a_{i_0j}+ \sum_{i_0\in I_0}k_{ji_0} a_{i_0i});$$

\item (common additive structure of small size) all but $O_C(n^\ep)$ entries $a_{ij}'$ belong to $Q$.
\end{itemize}
\end{theorem}

We remark that the common structure $Q$ is not yet a GAP, as the coefficients are rational, instead of being integers. It is desirable to have an analogue of Theorem \ref{theorem:ILO:optimal} with common structure as a genuine GAP with optimal parameters (see for instance \cite[Conjecture 1]{C} for a precise conjecture for bilinear forms.) For counting purposes, this inverse theorem is sufficiently strong.

\section{Application:  The least singular value of a random matrix} 

For a matrix $A$, let $\sigma_n(A)$ denote its smallest singular value. It is well known that $\sigma_n(A) \ge 0$ and the bound is strict if and only if $A$ 
is non-singular. 
An important problem with many practical applications is to bound the least singular value of a non-singular  matrix
 (see \cite{GN, ST, ST1, TVsmooth, Versur, Edelman} for discussions).  The problem  of estimating the least singular value of a random matrix was first raised by Goldstine and von Neumann \cite{GN}  in the 1940s, 
with connection to their investigation of the complexity 
of inverting a matrix.

To answer Goldstine and von Neumman's question, Edelman \cite{Edelman} 
computed the distribution of the LSV of  the random matrix $M_n^{\Gau} $ of size $n$ with iid standard gaussian entries, and showed
that for all fixed $t >0$

$$ \P( \sigma_n( M_n^{\Gau}  \leq t n^{-1/2}  ) = \int_0^{t }  \frac{1+\sqrt{x}}{2\sqrt{x}} e^{-(x/2 + \sqrt{x})}\ dx + o(1) =  t - \frac{1}{3} t^{3 }  +O(t^4) +o(1) . $$

He conjectured that this distribution is universal (i.e., it must hold for other models of random matrices, such as Bernoulli).

More recently, in their study of smoothed analysis of the simplex 
method,  Spielman and 
Teng \cite{ST, ST1} showed  that for any $ t>0$ ( $t$ can go to $0$ with $n$) 

\begin{equation}
\P (\sigma_n (M_n^{\Gau} )  \le t n^{-1/2} )  \le t.
\end{equation} 

\noindent They conjectured that a slightly adjusted bound holds in the Bernoulli case  \cite{ST}
\begin{equation}
\P (\sigma_n (M_n^{\Ber} )  \le t )  \le t n^{1/2} + c^{n},
\end{equation} 

\noindent where $0 < c < 1$ is a constant. The term $c^n$ is needed as $M_n^{\Ber}$ can be singular with exponentially small probability.

Edelman's conjecture has been proved by Tao and the second author 
in \cite{TVhard}. This work also confirms Spielman and Teng's conjecture for the case $t$ is fairly large; $t \ge n^{-\delta}$ for some small constant 
$\delta >0$. For $t \ge n^{-3/2} $, Rudelson in \cite{Rudannal}, making use of Hal\'asz' machinery from \cite{H}, obtained a strong bound with an extra (multiplicative) constant factor.
 In many applications,  it is important to be able to treat  even smaller $t$.  
As a matter of fact,  in applications what one usually needs is  the  probability bound to be  very small, but this requires one to set  $t$ very small automatically.

   In the last few years, thanks to the development of inverse theorems, one can 
now prove very strong bound for almost all range of $t$. 

Consider a matrix $M$ with row vectors $X_i$ and singular values $\sigma_1 \ge \dots \ge \sigma_n$. Let $d_i$ be the distance from $X_i$ to the hyperplane formed by the other $n-1$ rows.  There are several ways to exhibit 
a direct relation between the $d_i$ and $\sigma_i$. For instance,  Tao and the second showed \cite{TVcir} 

\begin{equation} 
d_1 ^{-2} +\dots + d_n^{-2}  = \sigma_1^{-2} + \dots + \sigma_n^{-2}. 
\end{equation} 

A technical relation, but in certain applications more effective,  is \cite[Lemma 3.5]{RV}. 

From this, it is clear that if one can bound the $d_i$ from below with high probability, then one can do the same for $\sigma_n$. Let $v =(a_1, \dots, a_n) $ be the normal vector of the hyperplane formed
by $X_2, \dots, X_n$ and $\xi_1, \dots, \xi_n$ be the coordinates of $X_1$,  then 

$$ d_1 = | a_1 \xi_1 + \dots a_n \xi_n | . $$

Thus, the probability that $d_1$ is small is exactly the small probability for the multi-set $A= \{a_1, \dots, a_n \}$.  If this probability is large,  then the inverse theorems tell us that  
the set $A$ must have strong additive structure. However, 
$A$ comes as the normal vector of a random hyperplane, so the probability that it has any special structure is very small (to quantify this we can use the counting theorems such as 
Theorem \ref{theorem:betanet}). Thus, we obtain, with high probability, a lower bound on all  $d_i$. In principle, one can use this to deduce a lower bound for $\sigma_n$. 

Carrying out  the above plan requires certain extra ideas and some careful analysis. In \cite{TVinverse}, Tao and the second author managed to prove 

\begin{theorem} \label{theorem:TVsingbound} 
For any constant $A>0 $, there is a constant $B>0$  such that 

$$ \P( \sigma_n (M_n^{\Ber} ) \le n^{-B}) \le n^{-A } . $$ 
\end{theorem} 

The first inverse theorem, Theorem \ref{theorem:TV1}, was first proved in this paper,  as a step in the proof of Theorem \ref{theorem:TVsingbound}. 
In a consequent paper, Rudelson and Vershynin developed Theorem \ref{theorem:RV}, and used it,  in combination  with \cite[Lemma 3.5]{RV}  and many other  ideas to show 

\begin{theorem} \label{theorem:RVsingbound} 
There is a constant $C >0$ and $0 < c <1$ such that for any $t>0$, 

$$ \P( \sigma_n (M_n^{\Ber} ) \le t n^{-1/2})  \le  t n^{1/2} + c^{n}  . $$ 
\end{theorem} 

This bound is sharp, up to the constant $C$. It also gives a new proof of Kahn-Koml\'os-Szemer\'edi bound on the singularity 
probability of a random Bernoulli matrix (see Section \ref{section:singularity2}).  Both theorems hold in more general setting.

In practice, one often works with random matrices of the type $A+M_n$ where $A$ is deterministic and $M_n$ has iid entries. 
(For instance, in their works on smoothed analysis, Spielman and Teng used this to model a large data matrix  perturbed by random noise.) They proved in \cite{ST} 

\begin{theorem} \label{theorem:STcondition} Let $A$ be an arbitrary $n$ by $n$ matrix.
 Then for any $ t >0$,
$$\P( \sigma_n (A+M_n^{\Gau})  \le t n^{-1/2}   ) = O(t) . $$
\end{theorem}

One may ask whether there is an analogue of Theorem \ref{theorem:RVsingbound} for this model. The answer is, somewhat surprisingly, negative. 
An analogue of the weaker Theorem \ref{theorem:TVsingbound} is, however, available, assuming that 
$\|A\|$ is bounded polynomially in $n$. For  more discussion on this model, we refer to \cite{TVsmooth}. 
For applications in Random Matrix Theory (such as the establishment of the Circular Law) and many related results, we 
refer to  \cite{TVbull,TVK,TVcir,GT, Pajor-log,Chafai, Versur} and the references therein.

\section{Application: Strong bounds on the singularity problem--the non-symmetric case}  \label{section:singularity2} 

We continue to discuss the singularity problem from  Section \ref{section:singularity1}.  The first exponential bound on $p_n$ was proved by Kahn, Koml\'os and Szemer\'edi 
\cite{KKSz}, who showed that $p_n \le .999^n$. In \cite{TVdet}, Tao and the second author simplified the proof and got a slightly improved bound $.952^n$. A more notable improvement 
which pushed the bound to $(3/4+o(1))^n$ was obtained in a subsequent paper \cite{TVsing}, which combined Kahn et. al. approach with an inverse theorem.  The best current bound is 
$(1/\sqrt 2 +o(1))^n$ by Bourgain, Vu and Wood \cite{BVW}. The proof of this bound still relied heavily on the approach from \cite{TVsing} (in particular it used the same inverse theorem), but 
added a new twist which made the first part of the argument more effective. 

In the following, we tried to present the 
approach from \cite{KKSz} and \cite{TVsing}. 
Similar to the proof in Appendix \ref{section:discrete:proof}, we  first embed the problem in a finite field $\F =\F_p$, where $p$ is a very large prime. 
Let   $\{-1,1\}^n \subset \F^n$ be the discrete unit cube in
$\F^n$.  We let $X$ be the random variable taking values in
$\{-1,1\}^n$ which is distributed uniformly on this cube (thus
each element of $\{-1,1\}^n$ is attained with probability
$2^{-n}$).  Let $X_1, \ldots, X_n \in \{-1,1\}$ be $n$ independent
samples of $X$.  Then
$$ p_n := \P( X_1, \ldots, X_n \hbox{ linearly dependent} ).$$

For each linear subspace $V$ of $\F^n$, let $A_V$ denote the event that
$X_1,\ldots,X_n$ span $V$.
Let us call a space $V$ \emph{non-trivial} if it is spanned by the set $V \cap \{-1,1\}^n$.
Note that $\P(A_V) \neq 0$ if and only if $V$ is non-trivial.
Since every collection of $n$ linearly dependent
vectors in $\F^n$ will span exactly one proper subspace $V$ of
$\F^n$, we have
\begin{equation}\label{pn-spanned}
 p_n = \sum_{V \hbox{ a proper non-trivial subspace of } \F^n} \P( A_V ).
\end{equation}

It is not hard to show that  the dominant contribution to this sum came from the hyperplanes:
$$ p_n = 2^{o(n)} \sum_{V \hbox{ a non-trivial hyperplane in } F^n} \P( A_V ).$$
Thus, if one wants to show $p_n \le (3/4 +o(1)) ^n$, it suffices to show 
$$ \sum_{V \hbox{ a non-trivial hyperplane in } \F^n} \P( A_V ) \leq (3/4 + o(1))^n.$$

The  next step is to partition the non-trivial hyperplanes $V$ into a number of classes, depending on
the number of $(-1,1)$ vectors in $V$. 

\begin{definition}[Combinatorial dimension]  Let $D := \{ d_\pm \in \Z/n:
1  \leq d_\pm \leq n\}$.
 For any $d_\pm \in D$, we define the \emph{combinatorial Grassmannian}
$\Gr(d_\pm)$ to be the set of all non-trivial
hyperplanes $V$ in $F^n$ with
\begin{equation}\label{discrete-dim}
2^{d_\pm - 1/n} < |V \cap \{-1,1\}^n| \leq 2^{d_\pm}.
\end{equation}
We will
refer to $d_\pm$ as the \emph{combinatorial dimension} of
$V$.
\end{definition}

It thus suffices to show that
\begin{equation}\label{pn-spanned-2}
 \sum_{d_\pm \in D} \sum_{V \in \Gr(d_\pm)} \P( A_V ) \leq (\frac{3}{4} + o(1))^n.
\end{equation}
It is therefore of interest to understand the size of the combinatorial Grassmannians $\Gr(d_\pm)$
and of the probability of the events $A_V$ for hyperplanes $V$ in those Grassmannians.

There are two easy cases, one when $d_\pm$ is fairly small and one where $d_\pm$ is fairly large.

\begin{lemma}[Small combinatorial dimension estimate]\label{kks-small} Let $0 < \alpha < 1$ be arbitrary.  Then
$$
\sum_{d_\pm \in D: 2^{d_\pm-n} \leq \alpha^n} \sum_{V \in \Gr(d_\pm)} \P( A_V )
\leq n\alpha^n.$$
\end{lemma}

\begin{proof} (of Lemma \ref{kks-small}) Observe that if $X_1, \dots, X_n$ span
$V$, then there are $n-1$ vectors among the $X_i$ which already
span $V$. By symmetry, we thus have

\begin{equation}\label{V-span}
 \P( A_V ) = \P( X_1,\ldots,X_n \hbox{ span } V)
 \leq n \P( X_1,\ldots,X_{n-1} \hbox{ span } V) \P( X \in V
 ).
 \end{equation}
On the other hand, if $V \in \Gr(d_\pm)$ and $2^{d_\pm-n} \leq \alpha^n$,
then $\P (X \in V) \le \alpha^n$ thanks to \eqref{discrete-dim}.  Thus we have
$$  \P( A_V )
 \leq n \alpha^n \P( X_1,\ldots,X_{n-1} \hbox{ span } V).$$
Since $X_1,\ldots,X_{n-1}$ can span at most one space $V$, the claim follows.
\end{proof}

\begin{lemma}[Large combinatorial dimension estimate]\label{kks-large} We have
$$
\sum_{d_\pm \in D: 2^{d_\pm-n} \geq 100 /\sqrt{n} } \,\,\,\,
\sum_{V \in \Gr(d_\pm)} \P( A_V ) \leq (1 + o(1))n^2
2^{-n}.$$
\end{lemma}

This proof uses  Theorem \ref{theorem:Erdos1} and  is left as an exercise; consult \cite{KKSz, TVsing} for details. 
The heart of the matter is  the following, somewhat more difficult, result.

\begin{proposition}[Medium combinatorial dimension estimate]\label{medium-dim}  Let $0 < \ep_0 \ll 1$,
and let $d_\pm \in D$ be such that $(\frac{3}{4} + 2\ep_0)^n < 2^{d_\pm-n} < \frac{100}{\sqrt{n}}$. Then we have
$$\sum_{V \in \Gr(d_\pm)} \P( A_V ) \leq o(1)^n,$$
where the rate of decay in the $o(1)$ quantity depends on $\ep_0$ (but not on $d_\pm$).
\end{proposition}

Note that $D$ has cardinality $|D| = O(n^2)$.  Thus if we combine
this proposition with Lemma \ref{kks-small} (with $\alpha :=
\frac{3}{4} + 2\ep_0$) and Lemma \ref{kks-large}, we see that we can bound the left-hand side of \eqref{pn-spanned-2}
by
$$ n (\frac{3}{4} + 2\ep_0)^n + n^2 o(1)^n + (1 + o(1))n^2 2^{-n} = (\frac{3}{4} + 2\ep_0 + o(1))^n.$$
Since $\ep_0$ is arbitrary, the upper bound $(3/4 +o(1))^n$ follows.

We now informally discuss the proof of Proposition \ref{medium-dim}.
We start with the trivial bound
\begin{equation}\label{trivial-bound}
 \sum_{V \in \Gr(d_\pm)} \P( A_V ) \le 1
\end{equation}
that arises simply because any vectors $X_1,\ldots,X_n$ can span at most one space $V$. To improve upon this trivial bound,
the key innovation in \cite{KKSz} is to replace $X$ by another
random variable $Y$ which tends to be more concentrated on
subspaces $V$ than $X$ is.  Roughly speaking, one seeks the
property
\begin{equation}\label{c-compare}
 \P( X \in V ) \leq c \P( Y \in V)
\end{equation}
for some absolute constant $0 < c < 1$ and for all (or almost all)
subspaces $V \in \Gr(d_\pm)$.  From this property, one expects (heuristically, at least)
\begin{equation}\label{c-compare-iter}
 \P(A_V) = \P( X_1,\ldots,X_n \hbox{ span } V ) \leq c^n \P( Y_1,\ldots,Y_n \hbox{ span } V ),
\end{equation}
where $Y_1,\ldots,Y_n$ are iid samples of $Y$, and then by
applying the trivial bound \eqref{trivial-bound} with $Y$ instead
of $X$, we would then obtain a bound of the form $\sum_{V \in \Gr(d_\pm)} \P( A_V ) \leq c^n$,
at least in principle.  Clearly, it will be desirable to make $c$ as small as possible; if we can make $c$
arbitrarily small, we will have established Proposition \ref{medium-dim}.

The random variable $Y$ can be described as follows.  Let $0 \leq
\mu \leq 1$ be a small absolute constant (in \cite{KKSz} the value
$\mu = \frac{1}{108} e^{-1/108}$ was chosen), and let
$\eta^{(\mu)}$ be a random variable taking values in $\{-1,0,1\}
\subset F$ which equals 0 with probability $1-\mu$ and equals $+1$
or $-1$ with probability $\mu/2$ each.  Then let $Y :=
(\eta^{(\mu)}_1,\ldots,\eta^{(\mu)}_n) \in F^n$, where
$\eta^{(\mu)}_1,\ldots,\eta^{(\mu)}_n$ are iid samples of
$\eta^{(\mu)}$.  By using a Fourier-analytic argument of
Hal\'asz \cite{H}, a bound of the form
$$ \P( X \in V) \leq C \sqrt{\mu} \P(Y \in V)$$

was shown in \cite{KKSz}, where $C$ was an absolute constant (independent of
$\mu$), and $V$ was a hyperplane which was \emph{non-degenerate}
in the sense that its combinatorial dimension was not too close to $n$.  For $\mu$ sufficiently small,
one then obtains \eqref{c-compare} for some $0 < c < 1$, although one cannot make $c$ arbitrarily small without
shrinking $\mu$ also.

There are however some technical difficulties with this approach,
arising when one tries to pass from \eqref{c-compare} to
\eqref{c-compare-iter}.  The first problem is that the random
variable $Y$, when conditioned on the event $Y \in V$, may
concentrate on a lower dimensional subspace on $V$, making it
unlikely that $Y_1,\ldots,Y_n$ will span $V$.  In particular, $Y$
has a probability of $(1-\mu)^n$ of being the zero vector, which
basically means that one cannot hope to exploit \eqref{c-compare}
in any non-trivial way once $\P(X \in V) \leq (1-\mu)^n$. However,
in this case $V$ has very low combinatorial dimension and
Lemma \ref{kks-small} already gives an exponential gain.

Even when $(1-\mu)^n < \P(X \in V) \leq 1$, it turns out that it is still
not particularly easy to obtain \eqref{c-compare-iter}, but one
can obtain an acceptable substitute for this estimate by only
replacing some of the $X_j$ by $Y_j$. Specifically,  one can try
to obtain an estimate roughly of the form
\begin{equation}\label{gammab} \P( X_1,\ldots,X_n \hbox{ span } V ) \leq c^{m}
\P( Y_1,\ldots,Y_m, X_1, \ldots, X_{n-m} \hbox{
span } V )
\end{equation}
where $m$ is equal to a suitably small multiple of $n$ (we will eventually take $m \approx n/100$).  Strictly speaking,
 we will also have to absorb an additional ``entropy'' loss of ${n \choose m}$ for technical reasons, though as we will be taking $c$
 arbitrarily small, this loss will ultimately be irrelevant.

The above approach (with some minor modifications) was carried out
rigorously in \cite{KKSz} to give the bound $p_n = O(.999^n)$ which
has been  improved slightly to $O(.952^n)$ in \cite{TVdet}, thanks to some simplifications. There
are two main reasons why the final gain in the base was relatively
small. Firstly, the chosen value of $\mu$ was small (so the
$n(1-\mu)^n$ error was sizeable), and secondly the value of $c$
obtained was relatively large (so the gain of $c^n$ or
$c^{(1-\gamma)n}$ was relatively weak).  Unfortunately, increasing
$\mu$ also causes $c$ to increase, and so even after optimizing
$\mu$ and $c$ one falls well short of the conjectured bound.

The more significant improvement to $(3/4 +o(1)) ^n$ relies on an inverse theorem. 
 To reduce all the other losses to $(\frac{3}{4} + 2\ep_0)^n$ for some small $\eps_0$,
we increase $\mu$ up to $1/4 - \ep_0/100$, at which point the arguments of Hal\'asz
and \cite{KKSz, TVdet}  give \eqref{c-compare} with $c = 1$. The value
$1/4$ for $\mu$ is optimal as it is the largest number satisfying
the pointwise inequality
$$ |\cos(x)| \leq  (1-\mu) + \mu \cos(2x) \hbox{ for all } x \in \R,$$
which is the Fourier-analytic analogue of \eqref{c-compare} (with $c=1$).
At first glance, the fact that $c=1$ seems to remove any utility to \eqref{c-compare}, as the above argument relied
on obtaining gains of the form $c^n$ or $c^{(1-\gamma)n}$. However, we can proceed further by subdividing the
collection of hyperplanes $\Gr(d_\pm)$ into two classes, namely the
\emph{unexceptional} spaces $V$ for which
$$ \P( X \in V) < \eps_1 \P(Y \in V)$$
for some small constant $0 < \eps_1 \ll 1$ to be chosen later (it will be much smaller than $\eps_0$), and the
\emph{exceptional} spaces for which
\begin{equation}\label{V-exceptional}
\eps_1 \P(Y \in V) \leq \P(X \in V) \leq \P(Y \in V).
\end{equation}
The contribution of the unexceptional spaces can be dealt with by
the preceding arguments to obtain a very small contribution (at
most $\delta^n$ for any fixed $\delta>0$ given that we set $\eps_1
=\eps_1(\gamma,\delta)$ suitably small), so it remains to consider
the exceptional spaces $V$.  

The key technical step is to show that there are very few exceptional hyperplannes (and thus their contribution is 
negligible). This can be done using the following inverse theorem (the way  the counting Theorem  \ref{counting} was proved using 
the inverse Theorem \ref{theorem:ILO:optimal}). 

Let $V \in \Gr(d_\pm)$ be an exceptional space, with a
representation of the form
\begin{equation}\label{V-rep}
 V = \{ (x_1,\ldots,x_n) \in F^n: x_1 a_1 + \ldots + x_n a_n = 0 \}
 \end{equation}
for some elements $a_1, \ldots, a_n \in F$. We shall refer to
$a_1,\ldots,a_n$ as the \emph{defining co-ordinates} for $V$.

\begin{theorem}\label{structure}  There is a constant $C=C(\ep_0, \ep_1)$
 such that the following holds. Let $V$ be a hyperplane in $
 \Gr(d_\pm)$ and $a_1,\ldots,a_n$ be its defining co-ordinates.
Then  there exist integers

\begin{equation}\label{rank-bound}
1 \leq r \leq C
\end{equation}
 and $M_1,\ldots,M_r \geq 1$ with the volume bound
\begin{equation}\label{m-max}
 M_1 \ldots M_r \leq C 2^{n - d_\pm}
\end{equation}
and non-zero elements $v_1, \ldots, v_r \in F$ such that the
following holds.

\begin{itemize}

\item (Defining coordinates lie in a progression) The
symmetric generalized arithmetic progression

$$ P := \{ m_1 v_1 + \ldots + m_r v_r: -M_j/2 < m_j < M_j/2 \hbox{ for all } 1 \leq j \leq r \}$$
is proper and contains all the $a_i$.

\item  (Bounded norm)  The $a_i$ have small $P$-norm:
\begin{equation}\label{p-mush}
 \sum_{j=1}^n \| a_j \|_P^2 \leq C
\end{equation}

\item  (Rational commensurability) The  set $\{
v_1,\ldots,v_r \} \cup \{a_1,\ldots,a_n\}$ is contained in the set
\begin{equation}\label{vr-rank}
\{ \frac{p}{q} v_1 : p, q \in \Z; q \neq 0; |p|, |q| \leq n^{o(n)}
\}.
\end{equation}

\end{itemize}

\end{theorem}

\section{Application:  Strong bounds on the singularity problem-the symmetric case } \label{section:singularity3} 

Similar to Conjecture \ref{notorious},  we raise

\begin{conjecture}  \label{notorious2}
$$p_n^{\sym} = (1/2 +o(1))^n.$$
\end{conjecture}

We are very far from this conjecture. Currently, no exponential upper  bound is known. 
The first superpolynomial bound   was obtained by the first author \cite{Ng} very recently. 

\begin{theorem}\cite{Ng}\label{theorem:sym:Ng} For any $C>0$ and $n$ sufficiently large
$$p_n^{\sym} \le cn^{-C}.$$ 
\end{theorem}

Shortly after, Vershinyn  \cite{Ver}  proved the following  better bound

\begin{theorem}\label{theorem:sym:V} There exists a positive constant $c$ such that 
$$p_n^{\sym}  =O(\exp(-n^c)).$$
\end{theorem}

Both proofs made essential use of inverse theorems. The first author used  the inverse quadratic Theorem \ref{theorem:ILO:quadratic} and Vershynin's proof used
Theorem \ref{theorem:RV} several times. 

In the following, we  sketched the  main ideas behind Theorem \ref{theorem:sym:Ng}. 
Let $\row=(\x_1,\dots,\x_n)$ be the first row of $M_{n}$, and $a_{ij}, 2\le i,j\le n$, be the cofactors of $M_{n-1}$ obtained by removing $\row$ and $\row^T$ from $M_n$. We have 

\begin{equation}\label{eqn:sym:formula}
\det(M_n)= \x_1^2 \det(M_{n-1}) + \sum_{2\le i,j \le n} a_{ij}\x_i\x_j.
\end{equation}

Recalling the proof of Theorem \ref{theorem:psym}
(see Section \ref{section:singularity1}). One first need to show that with high probability (with respect to $M_{n-1}$) 
a good fraction of the co-factors  $a_{ij}$ are nonzero.  Theorem \ref{theorem:LO:CTV} then yields that 

$$ \P_\row(\det(M_n)=0) \le n^{-1/8+o(1)} =o(1). $$  

To prove  Theorem \ref{theorem:sym:Ng}, we adapt the reversed approach, which, similar to the previous proofs, consists of an inverse statement and a counting step. 

\begin{enumerate}

\item (Inverse step). If $\P_\row(\det(M_n)=0|M_{n-1}) \ge n^{-O(1)}$, then there is a strong additive structure among the cofactors $a_{ij}$.

\vskip .1in 
  
\item (Counting step). With respect to $M_{n-1}$, a strong additive structure among the $a_{ij}$ occurs with negligible probability.   
\end{enumerate}

By \eqref{eqn:sym:formula}, one notices that the first step concentrates on the study of inverse Littlewood-Offord problem for quadratic forms $\sum_{ij}a_{ij}\x_i\x_j$. Roughly speaking, Theorem \ref{theorem:ILO:quadratic} implies that most of the $a_{ij}$ belong to a common structure. Thus, by extracting  the structure on one row of the array $A=(a_{ij})$, we obtain a vector which is orthogonal to the remaining $n-2$ rows of the matrix $M_{n-1}$. Executing the argument more carefully, we obtain the following lemma.

\begin{lemma}[Inverse Step]\label{lemma:normalvector:quadratic} Let $\ep<1$ and $C$ be positive constants.  Assume that $M_{n-1}$ has rank at least $n-2$ and that

$$\P_\row(\sum_{i,j}a_{ij}\x_i\x_j=0|M_{n-1})\ge n^{-C}.$$ 

Then there exists a nonzero vector $\Bu=(u_1,\dots,u_{n-1})$ with the following properties.

\begin{itemize}
\item All but $n^\ep$ elements of $u_i$ belong to a proper symmetric generalized arithmetic progression of rank $O_{C,\ep}(1)$ and size $n^{O_{C,\ep}(1)}$. 

\vskip .1in

\item $u_i \in \{p/q: p,q\in \Z, |p|,|q|=n^{O_{C,\ep}(n^\ep)}\}$ for all $i$.

\vskip .1in

\item $\Bu$ is orthogonal to $n-O_{C,\ep}(n^\ep)$ rows of $M_{n-1}$.
\end{itemize}

\end{lemma}

Let $\mathcal{P}$ denote the collection of all $\Bu$ satisfying the properties above. For each $\Bu\in \mathcal{P}$, let $\P_{\Bu}$ be the probability, with respect to $M_{n-1}$, that $u$ is orthogonal to $n-O_{C,\ep}(n^\ep)$ rows of $M_{n-1}$. The following lemma takes care of our second step.

\begin{lemma}[Counting Step]\label{lemma:total} We have
$$\sum_{\Bu\in \mathcal{P}}\P_{\Bu} = O_{C,\ep}((1/2)^{(1-o(1))n}).$$
\end{lemma}

The main contribution in the sum in Lemma \ref{lemma:total} comes from those $\Bu$ which have just a few non-zero components (i.e. compressible vectors). For incompressible vectors, we classify it into dyadic classes $\mathcal{C}_{\rho_1,\dots,\rho_{n-1}}$, where $\rho_i$ is at most twice and at least half the probability $\P(\x_1u_1+\dots+\x_uu_i=0)$. Assume that $\Bu\in \mathcal{C}_{\rho_1,\dots,\rho_{n-1}}$. Then by definition, as $M_{n-1}$ is symmetric, the probability $\P_{\Bu}$ is bounded by $\prod O(\rho_i)$. On the other hand, by taking into account the structure of generalized arithmetic progressions, a variant of Theorem \ref{counting} shows that the size of each $\mathcal{C}_{\rho_1,\dots,\rho_{n-1}}$ is bounded by $\prod_i O(\rho_i) n^{-1/2+o(1)}$. Summing $P_{\Bu}$ over all classes $\mathcal{C}$, notice that the number of these classes are negligible, one obtains an upper bound of order $n^{-(1-o(1))n/2}$ for the compressible vectors.

We remark that it is in the Inverse Step that we obtain the final bound $n^{-C}$ on the singular probability. In \cite{Ver}, Vershynin worked with a more general setting where one can assume a better bound. In this regime, he has been able to apply a variant of Theorem \ref{theorem:RV} to prove a very mild inverse-type result which is easy to be adapted for the Counting Step. As the details are complex, we invite the reader to consult \cite{Ver}.

\section{Application: Common roots of random polynomials} 

Let $d$ be fixed. With $\vec{j}_d=(j_1,\dots,j_d), j_i\in \Z^+$ and $|\vec{j}_d|=\sum j_i$, let $\xi_{\vec{j}_d}$ be iid copies of a  random variable $\xi$. Set $x^{\vec{j}_d}= \prod x_i^{j_i}$. 
Consider the random polynomial 

$$P(x_1,\dots,x_d)=\sum_{\vec{j}_d, |\vec{j}_d|\le n} \xi_{\vec{j}_d} x^{\vec{j}_d}$$

\noindent of degree $n$  in  $d$ variables. (Here $d$ is fixed and $n \rightarrow \infty$.) Random polynomials is a classical subject in analysis and probability and we refer to \cite{Samba} for a survey. 

In this section, we consider the following natural question. 
Let  $P_1,\dots,P_{d+1}$ be $d+1$ independent random polynomials, each have  $d$ variables and degree $n$. 

\begin{question}\label{question:common}
What is  the probability that $P_1,\dots,P_{d+1}$ have a common root ?
\end{question}

For short, let us denote the probability under consideration by $p(n,d)$

$$p(n,d):=\P(\exists x \in \C^d: P_i(x)=0, i=1,\dots,d+1).$$

When $\xi$ has continuous distribution, it is obvious that $p(n,d)=0$. However, the situation is less clear when $\xi$ has discrete distribution, even in the case $d=1$. Indeed, when $n$ is even and $P_1(x),P_2(x)$ are two independent random Bernoulli polynomials of one variable, then one has $\P(P_1(1)=P_2(1)=0)=\Theta(1/n)$ and $\P(P_1(-1)=P_2(-1)=0)=\Theta(1/n)$. Thus in this case  $p(n,1)= \Omega(1/n)$.

In a recent paper, Kozma and Zeitouni \cite{KZ} proved $p(n,d) =O(1/n)$, answering Question \ref{question:common} in the asymptotic sense. 

\begin{theorem}\label{theorem:KZ} For any fixed $d$  there exists a constant $c(d)$ such that  the following holds. 
Let $P_1\dots,P_{d+1}$ be $d+1$ independent random Bernoulli polynomials in $d$ variables and degree $n$. 

$$p(n,d) \le c(d)/n.$$
\end{theorem}

In the sequel, we will focus on the case $d=1$. This first case  already captures  some of the main ideas, especially the use of inverse theorems. 
 The reader is invited to consult  \cite{KZ} for further details.   

\begin{theorem}\label{theorem:KZ:1}
Let $P_1,P_2$ be two independent Bernoulli random polynomials in one variable of degree $n$. Then 
\[
p(n,1)=\begin{cases}
O(n^{-1}) & n \mbox{ even}\\
O(n^{-3/2}) & n \mbox{ odd}.
\end{cases}
\]

\end{theorem}

Notice that the bounds in both cases are sharp. To start the proof, first observe that,   because the coefficients of $P_1$ are $\pm 1$, all roots $x$ of $P_1$ have magnitude $1/2 < |x| <2$. Furthermore, $x$ must be an algebraic integer. 
 We will try to classify the common roots by their unique irreducible polynomial, relying on the following easy algebraic fact \cite{KZ}: 

\begin{fact}\label{fact:KZ:algebra}
For every $k$ there are only finitely many numbers whose irreducible polynomial has degree $k$ that can be roots of a polynomial of arbitrary degree with coefficients $\pm 1$.
\end{fact}

Now we look at the event of having common roots. Assume that $P_1$ is fixed (i.e. condition on $P_1$) and let $x_1,\dots,x_n$ be its $n$ complex roots. For each $x_i$, we consider the probability that $x_i$ is a root of $P_2(x)$. If $\P(P_2(x_i)=0)\le n^{-5/2}$ for all $i$, then $\P(\exists x\in \C: P_1(x)=P_2(x))=O(n^{-3/2})$, and there is nothing to prove. We now consider the case $\P(P_2(x_i)=0)\ge n^{-5/2}$ for some root $x_i$ of $P_1(x)$. Notice that 

$$\P(P_2(x_i)=0) = \P_{\xi_0,\dots,\xi_n}(\sum_{j=0}^n \xi_{j} x_i^j=0)=\rho(X),$$ 

where $X$ is the geometric progression $X=\{1,x_i,\dots,x_i^n\}$.

Now Theorem \ref{theorem:ILO:optimal} comes into play. As $\rho(X)\ge n^{-5/2}$, most of the terms of $X$ are additively correlated. On the other hand, as $X$ is a geometric progression, this is the case only if $x_i$ is a root of a bounded degree polynomial with well-controlled rational coefficients.

\begin{lemma}\label{lemma:ILO:poly}
For any $C >0$, there exists $n_0$ such that if $n>n_0$, and if 

$$\rho(X)\ge n^{-C},$$ 

where $X=\{1,x,\dots,x^n\}$. Then $x$ is an algebraic number of degree at most $2C$. 
\end{lemma}

\begin{proof}(of Lemma \ref{lemma:ILO:poly}) Set $\epsilon=1/(2C+2)$. Theorem \ref{theorem:ILO:optimal}, applied to the set $X$, implies that there exists a GAP $Q$ of rank $r$ and size $|Q|=O_C(n^{C-r/2})$ which contains at least $(2C+1)/(2C+2)$-portion of the elements of $X$. By pigeon-hole principle, there exists $2C+1$ consecutive terms of $X$, say $x^{i_0},\dots,x^{i_0+2C}$, all of which belong to $Q$. 

As  $|Q|\ge 1$, the rank $r$ of $Q$ must be at most $2C$. Thus there exist integral coefficients $m_1,\dots,m_{2C+1}$, all of which are bounded by $n^{O_C(1)}$, such that the linear combination $\sum_{i=0}^{2C}m_i x^{i_0+i}$ vanishes. In particular, it follows that $x$ is an algebraic number of degree at most $2C$. 
\end{proof}

We now prove Theorem \ref{theorem:KZ:1}. Write 

\begin{align*}
p(n,1)&= \P(\exists x\in \C: P_1(x)=P_2(x)=0)\\
&\le \P(P_1(1)=P_2(1)=0)+ \P(P_1(-1)=P_2(-1)=0)\\
&+ \P(\exists x \mbox{ of algebraic degree } 2,3,4,5: P_1(x)=P_2(x)=0)\\
&+ \P(\exists x \mbox{ of algebraic degree } \ge 6: P_1(x)=P_2(x)=0)\\
&=S_1+S_2+S_3.
\end{align*}

For the first term, it is clear that $S_1=\Theta(n^{-1})$ if $n$ is even, and $S_1=0$ otherwise. For the second term $S_2$, by Lemma \ref{fact:KZ:algebra}, the number of possible common roots $x$ of algebraic degree at most 5 is $O(1)$, so it suffices to show that $\P(P_1(x)=P_2(x))=n^{-3/2}$ for each such $x$. On the other hand, by Lemma \ref{lemma:ILO:poly} we must have $\P(P_i(x)=0)\le n^{-3/4}$ because $x$ cannot be a rational number (i.e. algebraic number of degree one). Thus we have

$$\P(P_1(x)=P_2(x)=0)=\P(P_1(x)=0)\P(P_2(x)=0)\le n^{-3/2}.$$   

Lastly, in order to bound $S_3$ we first fix $P_1(x)$. It has at most $n$ roots $x$ of algebraic degree at least 6. For each of these roots, by Lemma \ref{lemma:ILO:poly}, $\P(P_2(x)=0)=O(n^{-5/2})$. Thus the probability that $P_2$ has at least a common root with $P_1$ which is an algebraic number of degree at least 6 is bounded by $n \times O(n^{-5/2})=O(n^{-3/2})$. As a result, $S_3=O(n^{-3/2})$.

\section{Application: Littlewood-Offord type bound for  multilinear forms and Boolean circuits} \label{section:boolean} 

Let $k$ be a fixed positive integer, and $p(\x_1,\dots,\x_n)=\sum_{S\in [n]^{\le k}} c_S\x_S$ be a random multi-linear polynomial of degree 
at most $k$, where $\xi_i$ are iid Bernoulli variables (taking values $\{0,1\}$ with equal probability) and $\x_S=\prod_{i\in S} \x_i$. As mentioned in Section \ref{section:quadratic1},  by generalizing the proof of Theorem   \ref{theorem:LO:CTV}, Costelo, Tao and the second author proved the following

\begin{theorem}\label{theorem:mul:CTV} 
Let $K$ denote the number of non-zero coefficients $c_S$, and set $m:=K/n^{k-1}$. Then for any real number $x$ we have

$$\P(p=x) = O\Big(m^{-\frac{1}{2^{(k^2+k)/2}}}\Big).$$
\end{theorem}

Using a finer analysis,  Razborov and Viola \cite{RaVi} improved the exponent $\frac{1}{2^{(k^2+k)/2}}$  to  $\frac{1}{2k 2^k}$.

\begin{theorem}\label{theorem:multi:disjoint}
Let $p(\x_1,\dots,\x_n)=\sum_{S\in [n]^{\le k}} c_S\x_S$ be a multi-linear polynomial of degree $k$, and assume that there exist $r$ terms $\x_{S_1},\dots, \x_{S_r}$ of degree $k$ each where the $S_i$ are mutually disjoint and 
$c_{S_i} \neq 0$. Then for any real number $x$ we have 

$$\P(p=x) = O(r^{-b_k}),$$

where $b_k=(2k2^k)^{-1}.$ 
\end{theorem}

One observes that $r=\Omega (m/k)$, where $m$ was defined in Theorem \ref{theorem:mul:CTV}. Indeed, assume that the collection $\{ S_1,\dots, S_r\}$ is maximal (with respect to disjointness). Then every set $S$ with $c_S \neq 0$, either  $\x_S$ has degree less than $k$ or  $S$ intersects one of the  $S_i$. Thus $K=O(rk n^{k-1})$, and so $r=\Omega (m/k)$.  
 
 It is a very interesting question (in its own right and for  applications) to improve the exponent further. 
 In the rest of this section, we are going to discuss Razborov and Viola's  main application of Theorem \ref{theorem:multi:disjoint}. 
 
 For two functions $f,g: \{0,1\}^n \rightarrow \R$,  one defines their correlation as 

$$\Cor_n(f,g):= \P(f(\xi_1,\dots,\xi_n) = g(\xi_1,\dots,\xi_n))- 1/2, $$ 

where $\xi_i$ are iid Bernoulli variables taking values $\{0,1\}$ with equal probability.

Most of the research in Complexity Theory has so far concentrated on the case in which both $f$ and $g$ are
Boolean functions (that is $f(x),g(x)\in \{0,1\}$). To incorporate into this framework arbitrary multivariate polynomials, one converts them to Boolean functions. There are two popular ways of doing this. For a polynomial $p$ with integer coefficients, define a Boolean function $b(x) = 1$ if $m|p(x)$, where $m$ is a given integer, and 0 otherwise. These functions $b$ are called {\it modular polynomials}. For arbitrary $p$, one can set $b(x) = 1$ if $p(x) > t$ for some given threshold $t$, and 0 otherwise. We refer to these functions $b$ as {\it threshold polynomials}. For further discussion on these polynomials, we refer the reader to \cite{MTT61,MU71}.
 
 It is an open problem to exhibit an explicit Boolean function $f:\{0,1\}^n \rightarrow \{0,1\}$ such that $\Cor_n(b,f)= o(1/\sqrt{n})$ for any modular polynomial $b$ 
 whose underlying polynomial $p$  has degree $\log_2 n$ (see \cite{Vio09}). The same problem is also open  for threshold polynomials.

In \cite{RaVi}, Razborov and Viola initiated a similar study for the correlation of multi-variable polynomials where any output  outside of $\{0,1\}$ is counted as an error. They highlighted the following problem.

\begin{problem}\label{prob:logn}
Exhibit an explicit Boolean function $f$ such that $\Cor_n(p,f) =
o(1/\sqrt{n})$ for any real polynomial $p: \{0,1\}^n \rightarrow \R$ of degree $\log_2 n$.
\end{problem}

It is well-known that analogies between polynomial approximations and matrix approximations are important and influential in theory and other areas like Machine Learning (see for instance \cite{She08}). Viewed under this angle, Razborov and Viola's model
is a straightforward analogy of matrix rigidity \cite{Val77} that still remains one of the main unresolved problems in the modern Complexity Theory. For further discussion and motivation, we refer  to \cite{RaVi} and the references therein. It is noted that
solving Problem \ref{prob:logn} is a pre-requisite for solving the corresponding open problem for threshold polynomials. Similarly, the special case of Problem \ref{prob:logn} when the polynomials have integer coefficients is a pre-requisite for solving the corresponding open problem for modular polynomials. As a quick application of Theorem \ref{theorem:multi:disjoint}, we demonstrate here a result addressing the question for lower degree polynomials.

\begin{theorem}\cite[Theorem 1.2]{RaVi}\label{theorem:multi:negative}
We have $\Cor_n(p, parity)\le 0$ for every sufficiently large $n$ and every
real polynomial $p:\{0,1\}^n \rightarrow \R$ of degree at most $\log_2  \log_2 n/2$.
\end{theorem}

\begin{proof}(of Theorem \ref{theorem:multi:negative})
First we suppose that the hypothesis of Theorem \ref{theorem:multi:disjoint} is satisfied with $r= \sqrt{n}$. Then the probability that the polynomial outputs a Boolean value is bounded by 

$$2\times O((1/\sqrt{n})^{\frac{1}{2k2^k}})\le 1/2,$$

where $k\le \frac{1}{2}\log_2 \log_2 n$.

Otherwise, we can cover all the terms of degree $k$ by $k\sqrt{n}$ variables. Freeze these variables and iterate. After at most $k$ iterations, either the hypothesis of Theorem \ref{theorem:multi:disjoint} is satisfied with $r=\sqrt{n}$ (and with smaller degree), in which case we would be done, or else we end up with a degree-one polynomial with $n-O(k^2)\sqrt{n}\ge 1$ variables, in which case the statement  is true by comparison with the parity function.  
\end{proof}

\section{Application: Solving Frankl and F\"uredi's  conjecture}\label{section:FF} 

In this section, we return to the discussion in Section \ref{section:HD} and give a  proof of  Conjecture \ref{conj:FF}  and a new proof 
for Theorem \ref{FF1}.   Both proofs are based on the following inverse theorem. 

\begin{theorem} \label{TV1} For any fixed $d$ there is a constant $C$ such that the following holds. 
Let $ A=\{ a_1, \dots, a_n\}$ be a multi-set of  vectors in $\R^d$ such that $p_{d,1,\Ber} (A) \ge C k^{-d/2} $. 
Then $A$ is ''almost'' flat. Namely, there is a  hyperplane $H$ such that  $\dist(a_i,H) \geq 1$ for at most  $k$ values of $i=1,\ldots,n$. \end{theorem}

The proof of this theorem combines Esse\'en's bound (Lemma \ref{lemma:Esseen}) together with some geometric arguments. For details, see \cite{TV:FF}; 
$\dist(a, H_i)$, of course, means the distance from $a$ to $H_i$. 


We  first prove Theorem \ref{FF1} by induction on the dimension $d$.  The case $d=1$ follows from Theorem \ref{Erd1}, so we assume that $d \geq 2$ and that the claim has already been proven for smaller values of $d$.  
It  suffices to prove the upper bound

$$p(d, R, Ber, n)  \leq (1+o(1)) 2^{-n} S(n,s).$$

Fix $R$, and let $\eps > 0$ be a small parameter to be chosen later.  Suppose the claim failed, then there exists $R > 0$ such that for arbitrarily large $n$,
 there exist a multi-set $A = \{a_1,\ldots,a_n\}$ of vectors in $\R^d$ of length at least $1$ and a  ball $B$ of radius $R$ such that
\begin{equation}\label{xv}
\P( S_A \in B ) \geq (1+\eps) 2^{-n} S(n,s).
\end{equation}

In particular, from Stirling's approximation one has
$$ \P(S_A \in B) \gg n^{-1/2}.$$

 Applying the pigeonhole principle, we can find a ball $B_0$ of radius $\frac{1}{\log n}$ such that
$$ \P(S_A \in B_0) \gg n^{-1/2} \log^{-d} n.$$

Set $k := n^{2/3}$.  Since $d \ge 2$ and $n$ is large, we have
$$ \P(S_A \in B_0) \geq C k^{-d/2}$$
for some  fixed constant $C$.  Applying Theorem \ref{TV1} (rescaling by $\log n$), we conclude that there exists a hyperplane $H$ such that $\dist(v_i,H) \leq 1/\log n$ for at least $n-k$ values of $i=1,\ldots,n$.

Let $V'$ denote the orthogonal projection to $H$ of the vectors $v_i$ with $\dist(v_i,H) \leq 1/\log n$.  By conditioning on the signs of all the $\xi_i$ with $\dist(v_i,H) > 1/\log n$, and then projecting the sum $X_V$ onto $H$, we conclude from \eqref{xv} the existence of a $d-1$-dimensional ball $B'$ in $H$ of radius $R$ such that
$$ \P( X_{V'} \in B' ) \geq (1+\eps) 2^{-n} S(n,s).$$
On the other hand, the vectors in $V'$ have magnitude at least $1-1/\log n$.  If  $n$ is sufficiently large depending on $d,\eps$ this contradicts the induction hypothesis
 (after rescaling the $V'$ by $1/(1-1/\log n)$ and identifying $H$ with $\R^{n-1}$ in some fashion; notice that the scaling changes $R$ slightly but does not change $s$, and also that the function 
 $2^{-n} S(n,s)$ is decreasing with $n$).  This concludes the proof of \eqref{panda}.

Now we turn to the proof of Conjecture \ref{conj:FF}. We can assume $s \ge 3$, as the remaining cases have already been treated (see Section \ref{section:HD}). If the conjecture failed, then there exist arbitrarily large $n$ for which there exist a multi-set $A = \{a_1,\ldots,a_n\}$ of vectors in $\R^d$ of length at least $1$ and a  ball $B$ of radius $R$  such that
\begin{equation}\label{xv-2}
\P( S_A \in B ) > 2^{-n} S(n,s).
\end{equation}

 By iterating the argument used to prove \eqref{panda}, we may find a one-dimensional subspace $L$ of $\R^d$ such that $\dist(v_i,L) \ll 1/\log n$ for at least $n-O(n^{2/3})$ values of $i=1,\ldots,n$.  
 By reordering, we may assume that $\dist(v_i,L) \ll 1/\log n$ for all $1 \leq i \leq n-k$, where $k = O(n^{2/3})$.

Let $\pi: \R^d \to L$ be the orthogonal projection onto $L$.  We divide into two cases.  
The first case is when  $|\pi(v_i)| > \frac{R}{s}$ for all $1 \leq i \leq n$.  We then use the trivial bound
$$\P( S_A \in B ) \leq \P( S_{\pi(V)} \in \pi(B) ).$$

If we rescale Theorem \ref{Erd1} by a factor slightly less than $s/R$, we see that
$$ \P( S_{\pi(V)} \in \pi(B )) \leq 2^{-n} S(n,s) $$
which contradicts \eqref{xv-2}.

In the second case, we assume $|\pi(v_n)| \leq R/s$.  We let $A'$ be the multi-set  $\{a_1,\ldots,a _{n-k} \}$, then by conditioning on the $\xi_{n-k+1},\ldots,\xi_{n-1}$ we conclude the existence of a unit ball $B'$ such that
$$ \P(S_{A'} + \xi_n a_n  \in B') \geq \P( S_A \in B ).$$

Let $x_{B'}$ be the center of $B'$.  Observe  that 
if $S_{V'} + \xi_n a_n \in B'$ (for any value of $\xi_n$) then  $|S_{\pi(V')} - \pi(x_{B'})| \leq R+ \frac{R}{s}$.  Furthermore, if  $|S_{\pi(V')} - \pi(x_{B'})| > \sqrt{R^2-1}$, then the parallelogram law shows that $S_{V'} + a_n$ and $S_{V'}-_n$ cannot both lie in $B'$, and so conditioned on  $|S_{\pi(V')} - \pi(x_{B'})| > \sqrt{R^2-1}$, the probability that $S_{V'} + \xi_n a_n \in B'$ is at most $1/2$.

 We conclude that
\begin{align*}
& \P(S_{A'} + \xi_n a_n  \in B') \\ &\leq  \P( |A_{\pi(A')} - \pi(x_{B'})| \leq \sqrt{R^2-1} )
+ \frac{1}{2}  \P( \sqrt {R^2-1} < |S_{\pi(V')} - \pi(x_{B'})| \leq R+\frac{R}{s} ) \\
&=\frac{1}{2} \Big(\P( |A_{\pi(A')} - \pi(x_{B'})| \leq \sqrt{R^2-1} )+ \P(|S_{\pi(A')} - \pi(x_{B'})| \leq R+\frac{R}{s} ) \Big).
\end{align*}

However, note that all the elements of $\pi(A')$ have magnitude at least $1-1/\log n$.  
Assume, for a moment,  that $R$ satisfies 
\begin{equation}\label{deltas}
\sqrt{R^2-1} < s-1 \le R <   R +\frac{R}{s}  < s.
\end{equation}
  From Theorem \ref{Erd1} (rescaled by $(1-1/\log n)^{-1}$), we conclude  that
$$
\P( |S_{\pi(A')} - \pi(x_{B'})| \leq \sqrt{R^2-1} ) \leq 2^{-(n-k)} S(n-k,s-1)$$
and
$$
\P( |\pi(S_{A'}) - \pi(x_{B'})| \leq R+\frac{R}{s} ) \leq 2^{-(n-k)} S(n-k,s).$$

On the other hand, by Stirling's formula (if $n$ is sufficiently large) we have
$$ \frac{1}{2}  (2^{-(n-k)} S(n-k,s-1)) + \frac{1}{2} 2^{-(n-k)} S(n-k,s) = \sqrt {\frac{2}{\pi}} \frac{s-1/2+o(1)}{n^{1/2}}$$
while
$$ 2^{-n} S(n,s) = \sqrt {\frac{2}{\pi}} \frac{s+o(1)}{n^{1/2}}$$
and so we contradict \eqref{xv-2}. 

\vskip2mm

  An inspection of the above argument shows that all we need on $R$ are the conditions \eqref{deltas}. To satisfy the first inequality in \eqref{deltas}, we need
  $R < \sqrt {(s-1)^2 +1}$.  Moreover, once $s-1 \le R < \sqrt{(s-1)^2+1}$, 
  one can easily check that  $R+\frac{R}{s} < s$
holds automatically for any $s \ge 3$,  concluding  the  proof.



\appendix

\section{Proof of Theorem  \ref{theorem:ILO:optimal}}\label{section:discrete:proof}

In this section, we sketch the proof of Theorem \ref{theorem:ILO:optimal}. 

{\it Embedding.} The first step is to  embed the problem into a finite field $\F_p$ for some prime $p$. In the case when the $a_i$ are integers, we simply take $p$ to be a large prime
(for instance $p \ge 2^n (\sum_{i=1}^n |a_i| +1)$ suffices). 

If $A$ is a subset of a general torsion-free group $G$, we rely on the concept of Freiman isomorphism. Two sets $A,A'$ of additive groups $G,G'$ (not necessarily
torsion-free) are {\it Freiman-isomorphism of order $k$} (in generalized form) if there is an
bijective  map $f$ from $A$ to $A'$ such that $f(a_1)+\dots +f(a_k) =
f(a_1')+\dots +f(a_k')$ in $G'$ if and only if $a_1+\dots+a_k =
a_1'+\dots +a_k'$ in $G$, for any subsets $\{a_1, \dots, a_k \} \subset A; \{a_1', \dots, a_k' \} \subset A'$. 

The following theorem allows us to pass from an arbitrary torsion-free group to $\Z$ or cyclic groups of prime order (see \cite[Lemma 5.25]{TVbook}).

\begin{theorem}\label{theorem:Freimaniso} Let $A$ be a finite subset of a torsion-free additive group $G$. Then for any integer $k$ the following holds.

\begin{itemize}

\item there is a Freiman isomorphism $\phi$ : $A\rightarrow \phi(A)$ of order $k$ to some finite subset $\phi(A)$ of the integers $\Z$; 
\vskip .1in
\item more generally, there is a map $\phi$ : $A\rightarrow \phi(A)$ to some finite subset $\phi(A)$ of the integers $\Z$ such that

$$a_1+\dots +a_i = a_1'+\dots + a_j' \Leftrightarrow \phi(a_1)+\dots +\phi(a_i) = \phi(a_1')+\dots \phi(a_j')$$

for all $i,j\le k$. 
\end{itemize}

The same is true if we replace $\Z$ by $\F_p$, if $p$ is sufficiently large depending on $A$.
\end{theorem}

Thus instead of working with a subset $A$ of a torsion-free group, it is sufficient to work with subset of $\F_p$, where $p$ is large enough.  From now on, we can assume that $a_i$ are elements of $\F_p$ for some large prime $p$. We view elements of $\F_p$ as integers between $0$ and $p-1$. We use the short hand $\rho$ to denote $\rho (A)$. The next few steps are motivated 
by Hal\'asz' analysis in \cite{H}.

{\it Fourier Analysis.} The main advantage of working in $\F_p$ is that one can make use of discrete   Fourier analysis. Assume that $$\rho= \rho(A)=\P( S=a), $$ for some $a  \in \F_p$.
Using the standard notation $e_p(x)$ for $\exp(2\pi \sqrt{-1} x/p )$, we have

\begin{equation}\label{eqn:fourier1} \rho= \P(S=a)= \E \frac{1}{p} \sum_{t\in \F_p} e_p (t (S-a)) = \E \frac{1}{p} \sum_{t\in \F_p} e_p (t S) e_p(-t a) .\end{equation}

By independence

\begin{equation} \label{eqn:fourier2}  \E e_p(t S) = \prod_{i=1}^n e_p(t \x_i a_i)= \prod_{i=1}^n \cos \frac{2\pi t a_i}{p}.  \end{equation}

It follows that

\begin{equation} \label{eqn:fourier3} \rho  \le \frac{1}{p} \sum_{t \in \F_p} \prod_i |\cos \frac{2 \pi a_i t}{p}  |  = \frac{1}{p} \sum_{t \in \F_p} \prod_i |\frac{\cos  \pi a_i t}{p}  | , \end{equation}

where we made the change of variable $t \rightarrow t/2$ (in $\F_p$) to obtain the last identity. 

 By convexity, we have that   $|\sin  \pi z | \ge 2 \|z\|$ for any $z\in \R$, where $\|z\|:=\|z\|_{\R/\Z}$ is the distance of $z$ to the nearest integer. Thus,

\begin{equation} \label{eqn:fourier3-1}| \cos \frac{\pi x}{p}|  \le  1- \frac{1}{2} \sin^2 \frac{\pi x}{p}  \le 1 -2 \|\frac{x}{p} \|^2  \le \exp( - 2\|  \frac{x}{p} \| ^2 ) ,\end{equation}
where in the last inequality we used that fact that $1-y \le \exp(-y)$ for any $0 \le y \le 1$.

Consequently, we obtain a key inequality

\begin{equation} \label{eqn:fourier4}
\rho \le \frac{1}{p} \sum_{t \in \F_p} \prod_{i}|\cos \frac{ \pi a_i t}{p}  | \le  \frac{1}{p} \sum_{t \in F_p} \exp( - 2\sum_{i=1}^n  \| \frac{a_i t}{p}  \| ^2).
\end{equation}

{\it Large level sets.}  Now we consider the level sets $S_m:=\{t| \sum_{i=1}^n  \| a_i t/p \| ^2 \le m  \} $.  We have

$$n^{-C} \le \rho  \le  \frac{1}{p} \sum_{t \in \F_p} \exp( -2 \sum_{i=1}^n  \| \frac{a_i t }{p} \| ^2) \le \frac{1}{p} + \frac{1}{p} \sum_{m \ge 1} \exp(-2(m-1)) |S_m| .$$

Since $\sum_{m\ge 1} \exp(-m) < 1$, there must be  is a large level set $S_m$ such that

\begin{equation} \label{eqn:level1} |S_m| \exp(-m+2) \ge  \rho  p. \end{equation}

In fact, since $\rho \ge n^{-C}$, we can assume that $m=O(\log n)$.

{\it Double counting and the triangle inequality.} By  double counting we have

$$ \sum_{i=1}^n \sum_{t \in S_m} \|\frac{a_i t}{p} \| ^2 =   \sum_{t \in S_m} \sum_{i=1}^n  \|\frac{a_i t}{p} \| ^2 \le m |S_m |.$$

So, for most $a_i$

\begin{equation} \label{eqn:double1} \sum_{t \in S_m} \|\frac{a_i t}{p}\|^2  \le \frac{C_0 m }{n} |S_m| \end{equation}  for some large constant $C_0$.

Set $C_0 = \eps^{-1}$. By averaging, the set of $a_i$ satisfying \eqref{eqn:double1}
 has size at least $(1-\eps)n$.  We call this set $A'$. The set $A\backslash A'$ has size at most $\eps n$ and this is the
exceptional set that appears in Theorem \ref{theorem:ILO:optimal}. In the rest of the proof, we are going to show that $A'$ is a dense subset of a proper GAP.

Since $\|\cdot \|$ is a norm, by the triangle inequality, we have  for any $a  \in k A' $

\begin{equation} \label{eqn:double2} \sum_{t \in S_m} \|\frac{a t}{p}\|^2  \le  k^2 \frac{C_0 m}{n} |S_m| . \end{equation}

More generally, for any $l  \le k $ and $a \in lA'$

\begin{equation} \label{eqn:double3} \sum_{t \in S_m} \|\frac{a t}{p}\|^2  \le  k^2 \frac{C_0 m}{n} |S_m| . \end{equation}

{\it Dual sets.}  Define  $S_m^{\ast} :=\{ a | \sum_{t \in S_m}  \|\frac{a t}{p}\|^2 \le \frac{1}{200} |S_m |\}$ (the constant $200$ is adhoc and any sufficiently large constant would do).
$S_m^{\ast} $ can be viewed as some sort of a {\it dual} set of $S_m$. In fact, one can show as far as cardinality is concerned, it does behave like a dual

\begin{equation} \label{eqn:dual1} |S_m^{\ast} |  \le \frac{8p}{|S_m |} . \end{equation}

To see this, define $T_a :=\sum_{t \in S_m} \cos \frac{2\pi a t}{p}$. Using the fact that $\cos 2\pi z \ge 1 -100 \|z\|^2 $ for any $z \in \R$, we have, for any $a \in S_m^{\ast}$

$$T_a \ge  \sum_{t \in S_m} (1- 100 \| \frac{at}{p} \|^2 ) \ge \frac{1}{2} |S_m |. $$

One the other hand, using the basic identity $\sum_{a \in \F_p} \cos \frac{2\pi  ax}{p} = p\I_{x=0} $, we have

$$\sum_{a \in \F_p} T_a^2 \le 2p |S_m| . $$

\eqref{eqn:dual1} follows from the last two estimates and averaging.

Set $k := c_1 \sqrt{\frac{n}{m}}$, for a properly chosen  constant $c_1= c_1(C_0) $. By \eqref{eqn:double3}
we have $\cup_{l=1}^k  l A'  \subset  S_m^{\ast} $. Set $A^{''} = A' \cup \{0\}$; we have
$k A^{''} \subset S_m^{\ast} \cup \{0\} $. This results in the critical bound

\begin{equation}  \label{eqn:dual2} |k A^{''} |  = O( \frac{p}{|S_m|}) = O(\rho^{-1} \exp(-m+2)) .  \end{equation}

The role of $\F_p$ is now no longer important, so we can view the $a_i$ as integers.
Notice that \eqref{eqn:dual2} leads us to a situation similar to that of Freiman's inverse result (Therem \ref{theorem:Freiman}). In that theorem, we have a bound on $|2A|$ and conclude 
that $A$ has a strong additive structure. In the current situation, $2$ is replaced by $k$, which can depend on $|A|$. 
We can, however, finish the job by applying the following variant of Freiman's inverse theorem.

\begin{theorem}[Long range inverse theorem, \cite{NgV}]\label{theorem:longrange}
Let $\gamma>0$ be constant. Assume that $X$ is a subset of a torsion-free group such that $0 \in X$ and $ |kX| \le k^\gamma|X|$ for some integer $k \ge 2$ that may depend on $|X|$. Then there is proper symmetric GAP $Q$ of rank $r=O(\gamma)$ and cardinality $O_{\gamma}( k^{-r} |kX| )$ such that $X \subset Q$.
\end{theorem}

One can prove Theorem \ref{theorem:longrange} by combining  Freiman theorem with some extra combinatorial ideas and several facts about GAPs. For full details we refer to 
\cite{NgV}. 

The proof of the continuous version, Theorem  \ref{theorem:ILO:continuous},  is similar. 
Given a real number $w$ and a variable $\xi$, we define the $\xi$-norm of $w$ by $\|w\|_{\xi} := (\E\|w(\xi_1-\xi_2)\|^2)^{1/2},$ where $\xi_1,\xi_2$ are two iid copies of $\xi$. We have the following variant of Lemma \ref{lemma:Esseen}.

\begin{equation}\label{lemma:upperboundforsmallball}
\rho_{r,\xi}(A)\le \exp(\pi r^2)\int_{\R^d}\exp(-\sum_{i=1}^n \|\langle a_i,z \rangle\|_{\xi}^2/2 - \pi \|z\|_2^2) \\dz.
\end{equation}

This will play the role of \eqref{eqn:fourier3} in the previous proof.  The next steps are similar and  we refer the reader to  \cite{NgV} for more details.

\section{Proof of Theorem \ref{SBP}} \label{section:proofcont2} 

We provide here a proof from \cite{RV-rec} (see also \cite{FS}).  This proof is also influenced by Hal\'asz' analysis from \cite{H}. 
The starting point is again Esse\'en's bound. Applying Lemma 
\ref{lemma:Esseen}, we obtain

\begin{equation}  \label{esseen}
  \rho_{d,\beta \sqrt{d},\xi}(A)
  \le C^d \int_{B(0, \sqrt{d})} \prod_{k=1}^n
  |\phi(\langle \theta, a_k\rangle/\beta)| \, d \theta,
\end{equation}

where $\phi$ is the characteristic function. 

Let $\xi'$ be an independent copy of $\xi$ and denote by $\bar{\xi}$ the symmetric random variable $\xi-\xi'$. Then we easily have $|\phi(t)| \le \exp( -\frac{1}{2} ( 1- \E \cos (2 \pi  t \bar{\xi}))).$

Conditioning on $\xi'$, the assumption $\sup_a\P(\xi\in B(a,1))
\le 1-b$ implies that $\P(|\bar{\xi}| \ge 1) \ge b$. Thus,

\begin{align*}
  1- \E \cos (2 \pi  t \bar{\xi})
  &\ge \P(|\bar{\xi}| \ge 1) \cdot
    \E \Big (1- \cos (2 \pi  t \bar{\xi}) \mid |\bar{\xi}| \ge 1 \Big )
  \\
  &\ge b \cdot \frac{4}{\pi^2}
    \E \Big (\min_{q \in \Z} | 2 \pi  t \bar{\xi} - 2 \pi q|^2 \mid |\bar{\xi}| \ge 1 \Big )
  \\
  &= 16 b \cdot
    \E \Big (\min_{q \in \Z} |  t \bar{\xi} -  q|^2
        \mid |\bar{\xi}| \ge 1 \Big ).
\end{align*}

Substituting of this into \eqref{esseen} and using Jensen's inequality, we get

\begin{align*}
  \rho_{d,\beta \sqrt{d},\xi}(A) &\le C^d \int_{B(0, \sqrt{d})}
    \exp \Big( -8b \E \Big (
      \sum_{k=1}^n \min_{q \in \Z} |\bar{\xi} \langle \theta, a_k\rangle /\beta  -  q|^2
      \; \Big| \; |\bar{\xi}| \ge 1 \Big ) \Big) \, d \theta
  \\
  &\le C^d \E \Big( \int_{B(0, \sqrt{d})}
    \exp \Big( -8b \min_{p \in \Z^n} \Big\| \frac{\bar{\xi}}{\beta} \, \theta \cdot a
- p \Big\|_2 \Big)
       \, d \theta
       \; \Big | \; |\bar{\xi}| \ge 1 \Big)
  \\
  &\le C^d \sup_{z \ge 1} \int_{B(0,\sqrt{d})} \exp(- 8 b f^2(\theta)) \; d\theta,
\end{align*}

where $f(\theta) = \min_{p \in \Z^n} \Big\| \frac{z}{\beta} \, \theta \cdot a
- p \Big\|_2$.

The crucial step is to bound the size of the {\em recurrence set}

$$
I(t) := \Big\{ \theta \in B(0,\sqrt{d}) : \; f(\theta) \le t \Big\}.
$$

\begin{lemma}           \label{lemma:size recurrence}
  We have
  $$
  \mu(I(t)) \le \Big( \frac{Ct\beta}{\gamma \sqrt{d}} \Big)^d, \qquad t < \alpha/2.
  $$
\end{lemma}

\begin{proof}(of Lemma \ref{lemma:size recurrence})
Fix $t < \alpha/2$. Consider two points $\theta', \theta'' \in I(t)$.
There exist $p', p'' \in \Z^n$ such that

$$
\Big\| \frac{z}{\beta} \, \theta' \cdot a - p' \Big\|_2 \le t, \quad
\Big\| \frac{z}{\beta} \, \theta'' \cdot a - p'' \Big\|_2 \le t.
$$

Let

$$
\tau := \frac{z}{\beta} \, (\theta' - \theta''), \quad p := p' - p''.
$$

Then, by the triangle inequality,

\begin{equation}                            \label{<2t}
  \|\tau \cdot a - p\|_2 \le 2t.
\end{equation}

Recall that by the assumption of the theorem, $\LCD_{\alpha,\gamma}(a) \ge \frac{\sqrt{d}}{\beta}$. Thus, by the definition of the least common denominator, either $\|\tau\|_2 \ge \frac{\sqrt{d}}{\beta}$ or 

\begin{equation}                            \label{>min}
  \|\tau \cdot a - p\|_2 \ge \min(\gamma\|\tau \cdot a\|_2,\alpha).
\end{equation}

In the latter case, since $2t <\alpha$, \eqref{<2t} and \eqref{>min} imply

$$
2t \ge \gamma\|\tau \cdot a\|_2 \ge \gamma \|\tau\|_2,
$$

where the last inequality follows from \eqref{super-isotropy}.

Thus we have proved that every pair of
points $\theta', \theta'' \in I(t)$ satisfies:

$$ 
\text{either} \quad \|\theta' - \theta''\|_2 \ge \frac{\sqrt{d}}{z} =: R
\quad \text{or} \quad \|\theta' - \theta''\|_2 \le \frac{2t\beta}{\gamma z} =: r.
$$

It follows that $I(t)$ can be covered by Euclidean balls of radii
$r$,  whose centers are $R$-separated in the Euclidean distance.
Since $I(t) \subset B(0,\sqrt{d})$, the number of such balls is at
most 

$$\frac{\mu(B(0, \sqrt{d} + R/2))}{\mu(B(0,  R/2))} =\Big( \frac{2\sqrt{m}}{R} + 1 \Big)^d \le \Big( \frac{3\sqrt{d}}{R} \Big)^d.$$

Summing these volumes, we obtain $\mu(I(t)) \le ( \frac{3Cr}{R})^m$.
\end{proof}

\begin{proof}(of Theorem \ref{SBP}) First, by the definition
of $I(t)$ and as $\mu(B(0,\sqrt{d})\le C^d$, we have

\begin{align}                           \label{outside recurrence}
\int_{B(0,\sqrt{m}) \setminus I(\alpha/2)} \exp(- 8b f^2(\theta)) \; d\theta
  &\le \int_{B(0,\sqrt{d})} \exp(-2b\alpha^2) \; d\theta \nonumber \\
  &\le C^d \exp(-2b\alpha^2).
\end{align}

Second, by using Lemma~\ref{lemma:size recurrence}, we have
\begin{align}                           \label{on recurrence}
\int_{I(\alpha/2)} \exp(- 8 b f^2(\theta)) \; d\theta
  &= \int_0^{\alpha/2} 16 b t \exp(-8 b t^2) \mu(I(t)) \; dt \nonumber \\
&\le 16 b \Big( \frac{C\beta}{\gamma \sqrt{d}} \Big)^d
    \int_0^\infty t^{d+1} \exp(-8 b t^2) \; dt \nonumber \\
  &\le \Big( \frac{C'\beta}{\gamma \sqrt{b}} \Big)^d \sqrt{d}
  \le \Big( \frac{C''\beta}{\gamma \sqrt{b}} \Big)^d.
\end{align}

Combining \eqref{outside recurrence} and \eqref{on recurrence}
completes the proof of Theorem~\ref{SBP}.
\end{proof}

\end{document}